\documentclass[10pt,psamsfonts]{amsart}
\usepackage{amsmath}
\usepackage{amsthm}
\usepackage{amssymb}
\usepackage{amscd}
\usepackage{amsfonts}
\usepackage{amsbsy}
\usepackage{graphicx}
\usepackage{subfigure,graphicx}
\usepackage{overpic}
\usepackage{color}
\usepackage{caption}%
\usepackage[font=small,labelfont=bf]{caption}
\usepackage[normalem]{ulem}
\usepackage[top=3cm,bottom=2cm,left=1.5cm,right=1.5cm]{geometry}

\usepackage{float}

\DeclareCaptionFont{tiny}{\tiny}

\usepackage[dvips]{psfrag}
\usepackage{array}
\usepackage{color}
\usepackage{epsfig}
\usepackage{url}
\usepackage{overpic}
\usepackage{tikz-cd}
\usepackage{enumitem}
\usepackage{hyperref}
\usepackage{soul}

\newtheorem {theorem} {Theorem}
\newtheorem {proposition} [theorem]{Proposition}
\newtheorem {corollary} [theorem]{Corollary}
\newtheorem {lemma}  [theorem]{Lemma}

\newtheorem {example} [theorem]{Example}
\newtheorem {remark} [theorem]{Remark}

\newtheorem {mtheorem} {Theorem}

\usepackage{tabularx}
\newcommand{\R}{\ensuremath{\mathbb{R}}}

\newcommand{\CC}{\mathcal{C}}

\newcommand{\CF}{\ensuremath{\mathcal{F}}}
\newcommand{\CG}{\ensuremath{\mathcal{G}}}

\newcommand{\CO}{\ensuremath{\mathcal{O}}}
\newcommand{\CZ}{\ensuremath{\mathcal{Z}}}

\newcommand{\CP}{\ensuremath{\mathcal{P}}}
\newcommand{\CQ}{\ensuremath{\mathcal{Q}}}
\newcommand{\ov}{\overline}

\newcommand{\la}{\lambda}
\newcommand{\ga}{\gamma}
\newcommand{\G}{\Gamma}

\newcommand{\f}{\varphi}
\newcommand{\al}{\alpha}
\newcommand{\be}{\beta}

\newcommand{\A}{\ensuremath{\mathcal{A}}}

\newcommand{\C}{\ensuremath{\mathcal{C}}}

\newcommand{\bx}{{\bf x}}

\newcommand{\bg}{{\bf g}}

\newcommand{\sgn}{\mathrm{sign}}

\newcommand{\De}{\Delta}

\newcommand{\inte}{\mathrm{int}}
\def\p{\partial}
\def\e{\varepsilon}

\usepackage{mathpazo}

\title[Stability of smooth branches of periodic solutions]
{On the stability of smooth branches of periodic solutions\\ for higher order perturbed differential equations}

\author[D.D. Novaes and M.R. C\^{a}ndido$^{1}$]
{Murilo R. C\^{a}ndido$^{1}$ \and Douglas D. Novaes$^{2}$}

\address{Departamento de Matem\'{a}tica - Instituto de Matemática, Estatística e Computação Científica (IMECC) - Universidade Estadual de Campinas (UNICAMP), Rua S\'{e}rgio Buarque de Holanda, 651, Cidade
Universit\'{a}ria Zeferino Vaz, 13083--859, Campinas, SP, Brazil}\email{$^{1}$candidomr@ime.unicamp.br,$^{2}$ddnovaes@unicamp.br}

\subjclass[2010]{34C25,34C29,34D20}

\keywords{periodic solutions, asymptotic stability properties, averaging theory, k-hyperbolic matrices}

\begin{document}
\allowdisplaybreaks

\maketitle

\begin{abstract}
The averaging method combined with the Lyapunov-Schmidt reduction provides sufficient conditions for the existence of periodic solutions of the following class of  perturbative $T$-periodic nonautonomous differential equations $x'=F_0(t,x)+\varepsilon F(t,x,\varepsilon)$. Such periodic solutions bifurcate from a manifold $\CZ$ of periodic solutions of the unperturbed system $x'=F_0(t,x)$. Determining the stability of this kind of periodic solutions involves the computation of eigenvalues of matrix-valued functions $M(\varepsilon)$, which can done using the theory of $k$-hyperbolic matrices. Usually, in this theory, a diagonalizing process of $k$-jets of $M(\varepsilon)$ must be employed and no general algorithm exists for doing that. In this paper, we develop an alternative strategy for determining the stability of the periodic solutions without the need of such a diagonalization process, which can work even when the diagonalization is not possible.  Applications of our result for two families of $4$D vector fields are also presented.
\end{abstract}

\section{Introduction}

 The Averaging Method is a classical tool for providing asymptotic estimates for solutions of non-autonomous differential equations given in the following standard form 
 \begin{equation}\label{eq:e1}
	x^\prime = \sum_{i=1}^k\e^i F_i(t,x)+\e^{k+1}R(t,x,\e),
\end{equation}
	where $F_i\colon\R\times \Omega\to \R^n$, $i=1,\ldots,k$, and $R\colon\R\times \Omega\times[0,\e_0]\to \R^n$ are smooth functions $T-$periodic in the variable $t$, being $\Omega$ an  open subset of $\R^n$ and $\e_0 > 0$. Such asymptotic estimates are given in terms of solutions of a ``truncated averaged equation''. For an introduction of the averaging theory, the interested reader is referred to \cite{SVM}. 

In particular, the averaging method has been proven to be a very useful tool for studying periodic solutions (see \cite{chicone06,guckenheimer13,SVM}). In short, it provides a sequence of functions  $\bg_i:\Omega\rightarrow \R^n$, $i=0,1,\ldots,k$, with $\bg_0=0$, each one called {\it $i$th-averaged function}, for which the following result holds:

\begin{theorem}\label{thm:intro}
Suppose that for some $\ell\in\{1,2,\ldots,k\}$,  $\bg_0=\ldots=\bg_{\ell-1}=0$ and $\bg_{\ell}\neq0$, and assume that $z^*\in\Omega$ is a simple zero of $\bg_{\ell}$. Then, the differential equation \eqref{eq:e1} has an isolated $T$-periodic solution $\f(t,\e)$ satisfying $\f(t,\e)\to z^*$ as $\e\to0$. In addition, if all the eigenvalues of $\partial \bg_{\ell}(z^*)$ has negative real part, then the periodic solution $\f(t,\e)$ is asymptotically stable. Otherwise, if $\partial \bg_{\ell}(z^*)$ has an eigenvalue with positive real part, then the periodic solution $\f(t,\e)$ is unstable.
\end{theorem}

In this work, we use the notation $\p$ for differentiation and $\p_x$ for partial derivative with respect to the variable $x$.

Theorem \ref{thm:intro}  can be proven using basically two arguments. Let $x(t,z,\e)$ be the solution of the differential equation \eqref{eq:e1}  with initial condition $x(0,z,\e)=z$. First, it can be seen that for  $\e>0$ sufficiently small, the solution $x(t,z,\e)$ satisfies the following relationship
$$
x(T,z,\e)=z+\e^\ell\, \bg_\ell(z)+\CO(\e^{\ell+1}),
$$ 
see for instance \cite{LNT14,novaes2020higher}. 
Second, if $z^*$ is a simple zero of  $\bg_\ell(z)$, the Implicit Function Theorem provides a unique smooth branch $z(\e)=z^*+\CO(\e)$ satisfying 
$$
x(T,z(\e),\e)=z(\e).$$
 Consequently, $x(t,z(\e),\e)$ is a $T$-periodic solution of \eqref{eq:e1} and its stability can be determined by the stability of the fixed point $z(\e)$ of the Poincar\'{e} map
 $$\Pi(z,\e)=x(T,z,\e)$$ 
 which, under suitable hypothesis, can be provided by the eigenvalues of the Jacobian matrix $\partial\bg_\ell(z^*)$.  In fact, since we have
$$
\p_z\Pi(z(\e),\e)=I_n+\e^\ell \partial\bg_\ell(z^*)+\CO(\e^{\ell+1}),
$$
one can easily see that, for each eigenvalue $\lambda$ of $\partial\bg_\ell(z^*)$ satisfying $\mathrm{Re}(\lambda)>0\, (<0)$ there is an eigenvalue of $\p_z\Pi(z(\e),\e)$ satisfying $|\lambda(\e)|>1\, (<1)$ for all $\e\in(0,\e_0]$.

In Theorem \ref{thm:intro}, if one cannot ensure whether $\partial \bg_{\ell}(z^*)$ either has all its eigenvalues with negative real part or has an eigenvalue with positive real part, the stability of the periodic solution provided by Theorem \ref{thm:intro} can still be investigate in terms of the eigenvalues of a smooth matrix-valued function of the following kind 
\[
M(\e)=\partial \bg_{\ell}(z(\e))+\CO(\e).
\] 

The theory of $k$-hyperbolic matrices provides suitable conditions such that the eigenvalues of $M(\e)$  can be described by means of the eigenvalues of its $m$-jet, for some $1\leq m\leq k-\ell$. This approach has been used in several works (see, for instance, \cite{dilna2009uniqueness,montgomery1986existence,M06,MR2,MR1}).

The reader familiar with the averaging method for detecting isolated periodic solutions, may be aware that the hypothesis of Theorem \ref{thm:intro} favours the detection of periodic solutions, which the stability can be  described by a direct analysis of the non-singular Jacobian matrix $\partial \bg_{\ell}(z^*)$. However, some recent generalizations of the classical averaging method weakened such a hypothesis (see, for instance, \cite{BL, CLN17, LNT14,Novaes2022, NovaesSilva2021}). Therefore, the detection of periodic solutions bifurcating from  non-simple zeroes of the averaged function $\bg_{\ell} (z)$ became increasingly common, see for instance \cite{candido22zero,candido2019new,candido2020periodic}. 

In this case, one of the methods for determining the stability of the periodic solutions consists into finding a matrix 
$$
T(\e)=T_0+\e T_1+\cdots+\e^m T_m, 
$$
such that 
$$
T(\e)M(\e)T^{-1}(\e)=\sum_{i=0}^k\e^i\Lambda_i+\e^{k+1}\Lambda(\e),
$$ 
where  the matrix $\Lambda_i$ is diagonal for each $i=0,\dots,k$ and $\sum_{i=0}^k\e^i\Lambda_i$ is hyperbolic for all $\e\in(0,\e_0]$. The stability of the periodic solution will be, then, the same as the stability of the truncated matrix $\sum_{i=0}^k\e^i\Lambda_i$. One of the difficulties of this method is the lack of a general formula for obtaining the matrix $T(\e)$ resulting in the necessity of  performing long computations in order to find it for each periodic solution to be studied, see \cite[Section $3.2$]{M06}.  Moreover, in some cases it may be impossible to diagonalize any jet of the matrix $M(\e)$.

\subsection{Setting the problem}\label{setting}
In this paper, we are interested in a more general family of differential equations for which the averaging method also provides results on periodic solutions, namely
\begin{equation}\label{s1}
x'=F_0(t,x)+\sum_{i=1}^k \e^i F_i(t,x)+\e^{k+1}R(t,x,\e),
\end{equation}
where $F_i\colon\R\times \Omega\to \R^n$, $i=0,1,\ldots,k$, and $R\colon\R\times \Omega\times[0,\e_0]\to \R^n$ are smooth functions, $T-$periodic in the variable $t$, being $\Omega$ an open subset of $\R^n$ and $\e_0 > 0$.  Let $x(t,z,\e)$ be the solution of \eqref{s1}, and $x(t,z,0)$ the solution of the unperturbed system \begin{equation}\label{eq:unp}
x'=F_0(t,x).
\end{equation}  We assume that:
\begin{itemize}
\item[{\bf H1.}] there exists an open bounded set  $V\subset \R^m$, $m\leq n$, and a $C^{k+1}$ function $\beta:\ov V\rightarrow \R^{n-m}$ such that each solution of the unperturbed system \eqref{eq:unp}, with initial condition at the manifold
\[
\CZ=\{z_{\alpha}=(\alpha,\beta(\alpha)):\, \al\in \ov V\}\subset\Omega,
\]
 is $T$-periodic.

\item[{\bf H2.}] if $m<n$, define $\bg_0(z)=x(T,z,0)-z$ and let 
\begin{equation*}\label{Dg0}
\partial_z \bg_0(z_\al)=\left(\begin{array}{cc}*&\Gamma(\al)\\ *&\Delta(\al)\end{array}\right),
\end{equation*}
where $\Gamma(\al)$ and $\Delta(\al)$ are matrices with dimensions $m\times (n-m)$ and $(n-m)\times(n-m)$, respectively. We assume that $\det\left(\Delta(\al) \right)\neq 0$ for all $\al \in \ov{V}$. Notice that if $m=n$, the function $\beta(\al)$ does not appear and this hypothesis can be omitted.
\end{itemize}
In \cite{CLN17}, assuming Hypothesis {\bf H1} and {\bf H2}, the averaging method was combined with the Lyapunov-Schmidt reduction to provide sufficient conditions ensuring the existence of isolated $T$-periodic solutions of \eqref{s1} for $|\e|\neq0$ sufficiently small. In addition, the initial condition of such periodic solutions converges to the manifold $\CZ$ as $\e$ goes to $0$.  This result will be fairly discussed in Section \ref{Generic}. Accordingly, as our third main hypothesis, we assume that:
\begin{itemize}
\item[{\bf H3.}] for $\e\geq0$ sufficiently small and $\al^* \in V$, there exists a smooth branch of $T$-periodic solutions $\f(t,\e)$ of \eqref{s1}  satisfying $\f(0,\e)\to z_{\alpha^*}\in\CZ$,  as $\e\to 0$. 
\end{itemize}
Our  goal in this paper is to discuss how Theorem \ref{thm:intro} can be generalized in order to characterize the stability of periodic solutions satisfying hypothesis {\bf H3}.

\subsection{Main results} \label{main}
In order to state our main results, we need to establish some estimates for the solutions of the differential equation \eqref{s1}. Let $Y(t,z)$ be the principal fundamental matrix solution of the variational equation
\begin{equation}\label{vareq}
y'=\p_xF_0(t,x(t,z,0))y.
\end{equation}
By principal we mean $Y(0,z)=I_n$. Notice that, if one knowns the solution $x(t,z,0)$ of the unperturbed differential equation $x'=F_0(t,x)$ for every $z$ in a neighbourhood of $\CZ$, then we have $Y(t,z)=\partial_z x(t,z,0).$ 
For $\e>0$ sufficiently small and taking the initial point $\f(0,\e)$  in a neighbourhood of $\CZ$, we define
\begin{equation}\label{Yepsi}
Y_{\e}(t):=\p_z x(t,\f(0,\e),0)
\end{equation}
or still $Y_{\e}(t):=\widetilde Y(t,\f(0,\e))\widetilde Y(0,\f(0,\e))^{-1}$,  where $\widetilde Y(t,z)$ is any fundamental matrix solution of the variational equation \eqref{vareq}. The following Lemma provides an expansion in power series of $\e$ for the solutions of the differential equation \eqref{s1}.
\begin{lemma}[\cite{CLN17}]\label{popl1}
Let $x(t,z,\e)$ be the solution of the $\CC^{k+1}$  $T$-periodic  the differential equation \eqref{s1} such that
$x(0,z,\e)=z$. Then, there exists an open set $U_1 \subset\Omega$, with $\CZ\subset U_1$, and $\e_1>0$
such that, for each $z \in \ov U_1$ and $\e\in[-\e_1,\e_1]$, the solution $x(t, z, \e)$ is defined
on the interval $[0,t_{(z,\e)})$, with $t_{(z,\e)}>T$, and
\begin{equation*}\label{rfl}
x(t,z,\e)=x(t,z,0)+\sum_{i=1}^{k}\e^{i}\dfrac{y_i(t,z)}{i!}+\CO(\e^{k+1}),
\end{equation*}
where the functions $y_i$, for $1\leq i \leq k$,  are given recursively as
\begin{align}\label{smoothyi}
y_1(t,z)=&Y(t,z)\int_{0}^tY(s,z)^{-1}F_1(s,x(s,z,0))ds,\vspace{0.2cm} \nonumber\\
y_i(t,z)=&i!\,Y(t,z)\int_{0}^tY(s,z)^{-1}\Bigg(F_i(s,x(s,z,0))\\
&+\sum_{S'_i}\dfrac{1}{b_1!\,b_2!2!^{b_2}\cdots
b_{i-1}!(i-1)!^{b_{i-1}}}\p^{I'}F_{0}(s,x(s,z,0))\bigodot_{j=1}^{i-1}y_j(s,z)^{b_j}\nonumber\\
&+\sum_{l=1}^{i-1}\sum_{S_l} \dfrac{1}{b_1!\,b_2!2!^{b_2}\cdots
b_l!l!^{b_l}}\p^LF_{i-l}(s,x(s,z,0))
\bigodot_{j=1}^ly_j(s,z)^{b_j}\Bigg)ds.\nonumber
\end{align} 
\end{lemma}

Here,  $S_l$ is the set of all $l$--tuples of non--negative integers
$(b_1,b_2,\cdots,b_l)$ which are solutions of the equation
$b_1+2b_2+\cdots+lb_l=l$ and $L=b_1+b_2+\cdots+b_l$. Moreover
$S'_{l}$ denotes the  $(l-1)$--tuples of non--negative integers
satisfying $b_1+2b_2+\cdots+(l-1)b_{l-1}=l$ and
$L'=b_1+b_2+\cdots+b_{l-1}$. These expressions can also be provided in terms of  Bell's polynomials, see \cite{N17}. The notation above for $S_l$ and $S_l'$ will be used several times along the paper.

Accordingly, for $\e>0$ sufficiently small, the Poincar\'e map at time $T$ can be written in the form 
\begin{equation}\label{pm}
\Pi(z,\e)=x(T,z,0)+\sum_{i=1}^{k}\e^{i}\dfrac{y_i(T,z)}{i!}+\CO(\e^{k+1}),
\end{equation}
and its Jacobian matrix becomes 
$$
\partial_z\Pi(z,\e)=Y(T,z)+\CG_k(z,\e)+\CO(\e^{k+1}),
$$
with $$\CG_k(z,\e)=\sum_{i=1}^{k}\e^{i}\dfrac{\partial_z y_i(T,z)}{i!}.$$

Throughout this paper, $J_{\mu}^{\e}f(z,\e)$ stands for the $\mu$-jet of the function $f(z,\e)$ in the variable $\e$. When $f$ is a function of a single variable, we shall only use $J_{\mu}$ for denoting its $\mu$-jet.

Our first main result is concerned about the stability of the periodic solution $\f(t,\e),$ given by Hypothesis {\bf H3}, when $m=n$. In this case, the manifold of initial conditions of periodic solutions in Hypothesis {\bf H1} is given by $\CZ=\ov V$ and the function $\beta(\al)$ does not appear.

\begin{mtheorem}[Case $m=n$]\label{teoa}
Suppose that hypotheses {\bf H1}, {\bf H2}, and {\bf H3} hold for $m=n$.
Consider the smooth matrix-valued function 
\begin{equation}\label{Aepsi}
\A(\e)=\CG_k(\f(0,\e),\e)
\end{equation} and 
let $\ell\in\{1,\ldots, k-1\}$ be the power of the first non-vanishing coefficient $A_{\ell}$ of its expansion around $\e=0$.
Assume that  $A_{\ell}$ has no multiple eigenvalues on the imaginary axis.  Given an integer $\mu\leq k-\ell$, an eigenvalue $\la(\e)$ of $J_{\mu}\big(\e^{-\ell}\A(\e)\big)$ is $C^k$ provided that $\Re(\la(0))=0$. In addition, the following statements hold.
\begin{itemize}
\item[a.] Assume that each eigenvalue $\la(\e)$ of $J_{\mu}\big(\e^{-\ell}\A(\e)\big)$ satisfies either $\Re( \la(0))<0$, or $\Re( \la(0))=0$ and $\Re(J_{\mu} \la(\e))<0$ for $\e\in(0,\e_0)$. Then, for $\e>0$ sufficiently small, the periodic solution $\f(t,\e)$ is asymptotically stable.

\item[b.] On the other hand, assume that there exists an eigenvalue $\la(\e)$ of  $J_{\mu}\big(\e^{-\ell}\A(\e)\big)$ such that either $\Re( \la(0))>0$, or $\Re( \la(0))=0$ and $\Re(J_{\mu} \la(\e))>0$ for $\e\in(0,\e_0)$. Then, for $\e>0$ sufficiently small, the periodic solution $\f(t,\e)$ is unstable.
\end{itemize}
\end{mtheorem}

Our second main result is concerned about the stability of the periodic solution $\f(t,\e)$ when $m<n$.  Before stating the result, we need to establish some terminology. 
Recall that $\varphi(t,\e)$, given by Hypothesis {\bf H3}, satisfies $\varphi(0,\e)=z_{\al^*}+\CO(\e)$, where $z_{\al^*}=(\al^*,\beta(\al^*))$.  Moreover, by   {\bf H1} we have that $x(T,(\alpha,\beta(\alpha)),0)=(\alpha,\beta(\alpha))$, for every $\alpha\in\ov V$. Thus, taking \eqref{Yepsi} into account and computing the derivative of the above expression at $\al=\al^*$, we obtain
\begin{equation} \label{A}
Y_0(T)\left(\begin{array}{c}I_m\\ \partial\beta(\al^*)\end{array}\right)=\left(\begin{array}{c}I_m\\ \partial\beta(\al^*)\end{array}\right).
\end{equation}
On the other hand, from hypothesis {\bf H2} we have that
\begin{equation}\label{B} 
Y_0(T)=\left(\begin{array}{cc}*&\Gamma(\al^*)\\ *&I_{n-m}+\Delta(\al^*)\end{array}\right). 
\end{equation}
Consequently, from \eqref{A} and \eqref{B} one can deduce that 
\[Y_0(T)=
\left(\begin{array}{cc}I_m-\Gamma(\al^*)\p\beta (\al^*)&\Gamma  (\al^*)\\-\Delta (\al^*) \p\beta  (\al^*) &I_{n-m}+\Delta  (\al^*) \end{array}\right).
\]

Now we are in a position to state our last fundamental hypothesis:
\begin{itemize}
\item[{\bf H4.}]   Assume that the matrix 
$$
I_{n-m}+\Delta(\al^*)-\p\beta (\al^*) \Gamma(\al^*),
$$
has no multiple eigenvalues on the unitary circle and $1$ is not one of its eigenvalues.
\end{itemize} 
Notice that, hypothesis {\bf H3}  implies that $0$ is not eigenvalue of $\Delta(\al^*)-\p\beta (\al^*) \Gamma(\al^*)$ and consequently 
$$
\det(\Delta(\al^*)-\p\beta (\al^*) \Gamma(\al^*))\neq 0.
$$

In order to state our next result, we introduce the matrix 
\begin{equation}\label{ML}
L=\left(\begin{array}{cc}
I_m+(\Delta(\al^*)-\p\beta (\al^*)\Gamma(\al^*))^{-1} \Gamma(\al^*) \partial\beta(\al^*)& \left(\Delta(\al^*)-\p\beta (\al^*)\Gamma(\al^*)\right)^{-1}\Gamma(\al^*)\\-\p\beta (\al^*)&\quad I_{m-n}
\end{array}\right),
\end{equation}
which is well defined by Hypothesis {\bf H4}, and  block-diagonalize the matrix $Y_0(T)$  as follows
\begin{equation}\label{LYIL}
\begin{split}L\, Y_0(T)\,L^{-1}=
&\left(\begin{array}{cc}
I_m & 0\\
0& I_{n-m}+\Delta(\al^*)-\p\beta (\al^*)\Gamma(\al^*)\end{array}\right).
\end{split}
\end{equation}
Additionally, we define the smooth matrix-valued functions $A(\e)_{m\times m}$, $B(\e)_{m\times n-m}$, $C(\e)_{n-m\times m}$, and $D(\e)_{n-m \times n-m}$ as follows
\begin{equation}\label{mp1}
\begin{split}
&\left(\begin{array}{cc}A(\e)&B(\e)\\C(\e)&D(\e)\end{array}\right):=L\Big(Y_{\e}(T)-Y_0(T)+ \CG_k(\f(0,\e),\e)\Big) L^{-1},
\end{split}
\end{equation}
where that the matrices $A(\e)$, $B(\e)$, $C(\e)$, and $D(\e)$ vanish at $\e=0$. Taking  these functions into account  we define  the matrices 
\begin{equation}\label{NM}
\begin{aligned}
N(\e)&=I_{n-m}+\Delta(\al^*)-\p\beta (\alpha^*)\Gamma(\al^*)+D(\e),\\
M(\e)&=A(\e)-B(\e)\Big(\Delta(\al^*)-\p\beta (\al^*)\Gamma(\al^*)+D(\e)\Big)^{-1}C(\e).
\end{aligned}
\end{equation}
Consider the power expansion of $M(\e)$ around $\e=0$ and assume that $M_{\ell}$ is its first non-vanishing coefficient, i.e 
\begin{equation}\label{Mep}
 M(\e)=\e^{\ell} M_{\ell}+\CO(\e^{\ell+1}),\quad M_{\ell}\neq0, \mbox{ and } \quad \ell \leq k.
 \end{equation}

For  $\omega,\la\in\mathbb{C}$, we consider the following polynomials
\begin{equation}\label{cor}
\CP(\omega;\e)=\det\Big[N(\e)-\omega\, I_{n-m}-C(\e)\big((1-\omega)I_m+\e\, A(\e)\big)^{-1} B(\e)\Big]
\end{equation}
and
\begin{equation}\label{est}
\begin{aligned}
&\CQ(\la;\e)= \det\Big[ M(\e)-\e^\ell \lambda I_{m}-\e^\ell B(\e)\sum^\infty _{k=1}\left((\Delta(\al^*)-\p\beta (\al^*)\Gamma(\al^*)+D(\e))^k\lambda^k\e^{\ell(k-1)}\right)C(\e)\Big].
\end{aligned}
\end{equation}
The following theorem is our second main result. 
\begin{mtheorem}[Case $m<n$]\label{teob}
Suppose that hypotheses {\bf H1}--{\bf H4} hold for $m<n$. Assume that $M_{\ell}$ $(\ell\leq k)$  has no multiple eigenvalues on the imaginary axis.  Given positive integers $\mu_1,\mu_2\leq k$, let $ \omega_i(\e)$ for $i\in\{1,\ldots,i^*\}$, $1\leq i^*\leq n-m$, and $ \lambda_j(\e)$ for $j\in\{1,\ldots,j^*\}$, $1\leq j^*\leq m$, be, respectively, all the continuous branches of roots of $J_{\mu_1}^{\e}\CP(\omega;\e)$ and $J_{\mu_2+m\, \ell}^{\e}\CQ(\lambda;\e)$ satisfying $|\omega_i(0)|=1$ and $\Re(\la_j(0))=0$, if they exist. Then, $\omega_i(\e)$ and $\la_j(\e)$,  for $i\in\{1,\ldots,i^*\}$ and $j\in\{1,\ldots,j^*\}$, are $C^k$ functions for $\e\geq0$ small. In addition,
\begin{itemize}
\item[a.] the periodic solution $\f(t,\e)$ is asymptotically stable for $\e>0$ sufficiently small, provided that the following two conditions hold:
\begin{itemize}
\item[a1.] All the eigenvalues of $I_{n-m}+\Delta(\al^*)-\p\beta(\al^*)\Gamma(\al^*)$ are contained within the closure of the unitary disk and, for each $i\in\{1,\ldots,i^*\}$, $|J_{\mu_1} \omega_i(\e)|<1$  for $\e\in(0,\e_0);$
\item[a2.] All the eigenvalues of $M_{\ell}$ are contained within the closure of the left half-plane and, for each $j\in\{1,\ldots,j^*\}$, $\Re(J_{\mu_2} \lambda_j(\e))<0$  for $\e\in(0,\e_0)$.
\end{itemize}

\item[b.] On the other hand, the periodic solution $\f(t,\e)$ is unstable for $\e>0$ sufficiently small, provided that at least one of the following two conditions holds:
\begin{itemize}
\item[b1.] $I_{n-m}+\Delta(\al^*)-\p\beta(\al^*)\Gamma(\al^*)$ has eigenvalues outside the unitary disk or $M_{\ell}$ has eigenvalues within the right-half plane;
\item[b2.] there exists $i\in\{1,\ldots,i^*\}$ such that $|J_{\mu_1} \omega_i(\e)|>1$ for $\e\in(0,\e_0)$ or  there exists $j\in\{1,\ldots,j^*\}$ such that $\Re(J_{\mu_2} \lambda_j(\e))>0$  for $\e\in(0,\e_0)$. 
\end{itemize}
\end{itemize}
In case $\omega_i$'s (resp. $\lambda_j$'s) do not exist, the above statements still hold with no consideration regarding them.
\end{mtheorem}

 \subsection{Structure of the paper} Theorems \ref{teoa} and \ref{teob} are proven in Sections \ref{sec:teoa} and \ref{sec:teob}, respectively. In Section \ref{Generic}, some results from \cite{CLN17} are discussed and, based on that, asymptotic expansions for $\f(0,\e)$ are provided.  Finally, Section \ref{sec:app} is devoted to provide applications of Theorems \ref{teoa} and \ref{teob} for $4$D polynomial vector fields.

\section{Proof of Theorem \ref{teoa}}\label{sec:teoa}
 Consider the solution of the $T$-periodic non-autonomous differential equation \eqref{s1}, $x(t,z,\e)$. From Lemma \ref{popl1}, the Poincar\'e map $\Pi(z,\e)=x(T,z,\e)$ writes
 \[
\Pi(z,\e)=z+\sum_{i=1}^{k}\e^{i}\dfrac{y_i(T,z)}{i!}+\CO(\e^{k+1}).
\]
Consequently, the stability of the $T$-periodic solution $\varphi(t,\e)$ is given by the eigenvalues of the Jacobian matrix
\begin{equation*}
\partial_z\Pi(\varphi(0,\e),\e)=Id+\A(\e)+\CO(\e^{k+1}),
\end{equation*}
where $\A(\e)$ is given by \eqref{Aepsi}.

By hypothesis, $\A(\e)=\e^{\ell} A_{\ell}+\CO(\e^{\ell+1})$, thus the eigenvalues of $\partial_z\Pi(\varphi(0,\e),\e)$ have the form $\omega(\e)=1+\e^\ell \eta(\e)$, where $\eta(\e)$ is an eigenvalue of the matrix
\begin{equation*}\label{al}
\widetilde \A(\e):=\dfrac{\partial_z\Pi(\varphi(0,\e),\e)-Id}{\e^{\ell}}=\dfrac{\A(\e)}{\e^{\ell}}+\CO(\e^{k+1-\ell})=A_\ell+\e  R_{\mu}(\e)+\CO(\e^{\mu+1})
\end{equation*}
and  $R_{\mu}(\e)$ is a matrix function, polynomial in $\e$ of degree $\mu-1,$ with $\mu\leq k-\ell$, satisfying $$A_{\ell}+\e R_{\mu}(\e)=J_{\mu}\big(\e^{-\ell} \A(\e)\big).$$

Notice that, for $\e>0$ sufficiently small, $|\omega(\e)|<1$ (resp. $|\omega(\e)|>1$)  provided that $\Re(\eta(\e))<0$ (resp. $\Re(\eta(\e))>0$). Indeed,
\[|\omega(\e)|^2=|1+\e^\ell \eta(\e)|=1+2\e^\ell \Re\left(\eta(\e)\right)+\e^{2\ell}|\eta(\e)|^2.\]
Thus, in what follows, we study the sign of the real part of the eigenvalues of $\widetilde A(\e)$.

Let $\mathcal{B}(\eta;\e)$ and $\mathcal{B}_{\mu}(\eta;\e)$ be the characteristic polynomials of $\widetilde \A(\e)$ and $J_{\mu}\big(\e^{-\ell}\A(\e)\big)$, respectively. That is,
\begin{equation}
\begin{aligned}\label{eq:chapoly}
\mathcal{B}(\eta;\e)&=\det\left( A_\ell-\eta Id+\e  R_{\mu}(\e)+\CO(\e^{\mu+1})\right)\,\text{ and}\\
\mathcal{B}_{\mu}(\eta;\e)&=\det\left(A_{\ell}-\eta Id+\e R_{\mu}(\e)\right).
\end{aligned}
\end{equation}
Notice that 
$\mathcal{B}(\eta;0)=\mathcal{B}_{\mu}(\eta;0)=\mathcal{B}_0(\eta)$,
where 
$\mathcal{B}_0(\eta)=\det(A_\ell-\eta I_n)$ is the characteristic polynomial of $A_{\ell}$.

Let $n^u$ (resp. $n^s$) be the number of roots of $\mathcal{B}_0(\eta)$, counting multiplicity, with positive (resp. negativity) real part. Consider a simple closed curves $C^u\subset \{z\in\mathbb{C}:\,\Re(z)>0\}$ and $C^s\subset \{z\in\mathbb{C}:\,\Re(z)<0\}$ such that $\inte(C^u)$ contains  all the $n^u$ roots of $\mathcal{B}_0(\eta)$ with positive real part and  $\inte(C^s)$ contains all the $n^s$ roots of $\mathcal{B}_0(\eta)$ with negative real part. Notice that  there exists $\ov \e>0$ such that  $|\mathcal{B}_0(z)|>|\mathcal{B}(z;\e)-\mathcal{B}_0(z)|$ and  $|\mathcal{B}_0(z)|>|\mathcal{B}_{\mu}(z;\e)-\mathcal{B}_0(z)|$ for every $z\in C^u\cup C^s$. Applying Rouche's Theorem we have that, for every $\e\in[0,\ov \e]$, the polynomials $\mathcal{B}_0(\eta)$, $\mathcal{B}(\eta;\e)$, and $\mathcal{B}_{\mu}(\eta;\e)$ have the same number of roots inside $\inte(C^u)$ and $\inte(C^s)$, counting multiplicity. 

Now, denote by $n^c$ the number of roots of $\mathcal{B}_0(\eta)$ on the imaginary axis. Clearly, we have $n^c+n^u+n^s=n$. In what follows we study the roots of $\mathcal{B}(\eta;\e)$ and $\mathcal{B}_{\mu}(\eta;\e)$ close to the $n^c$ roots of $\mathcal{B}_0(\eta)$ on the imaginary axis.

Let $\eta^*$ be a root of $\mathcal{B}_0(\eta)$ on the imaginary axis. Since $\mathcal{B}_0(\eta)$ is the characteristic polynomial of $A_{\ell}$, we have by hypothesis that $\mathcal{B}_0'( \eta^*)\neq0$ and, consequently,
$\p_{\eta}\mathcal{B}(\eta^*;0)\neq0$ and $\p_{\eta}\mathcal{B}_{\mu}(\eta^*;0)\neq0$. Then, from the Implicit Function Theorem, there exists a unique root $\eta(\e)$ of $\mathcal{B}(\eta;\e)$ and a unique root $\la(\e)$ of $\mathcal{B}_{\mu}(\la;\e)$ satisfying $\eta(0)=\la(0)=\eta^*$. In addition, $\eta(\e)$ and $\la(\e)$ are  $\C^k$ in a neighborhood of $\e=0$. Furthermore,  taking \eqref{eq:chapoly}  into account, we see that $J_{\mu}\eta(\e)=J_{\mu}\la(\e)$ and, consequently, for $\e>0$ sufficiently small
\[
\sgn\left(\Re(\eta(\e))\right)=\sgn\left(\Re(J_{\mu}\eta(\e))\right)=\sgn\left(\Re(J_{\mu}\la(\e))\right),
\]
provide that $\Re(J_{\mu}\la(\e))\neq 0$.

Summarizing, we proved that:
\begin{itemize}
\item[-] for the $n^u$ (resp. $n^s$) eigenvalues $\la(\e)$ of $J_{\mu}\big(\e^{-\ell}\A(\e)\big)$, counting multiplicity, such that $\Re(\la(0))>0$ (resp. $\Re(\la(0))<0$), there exist $n^u$ (resp. $n^s$) eigenvalues  $\eta(\e)$ of $\widetilde \A(\e)$, also counting multiplicity, satisfying $\Re(\eta(\e))>0$ (resp. $\Re(\eta(\e))<0$), for $\e>0$ sufficiently small.

\item[-] for the $n^c$ eigenvalues $\la(\e)$ of $J_{\mu}\big(\e^{-\ell}\A(\e)\big)$ such that $\Re(\la(\e))=0$, there exist $n^c$ eigenvalues  $\eta(\e)$ of $\widetilde\A(\e)$ satisfying 
$$
\sgn(\eta(\e))=\sgn(J_{\mu}\la(\e)),
$$
for $\e>0$ sufficiently small.
\end{itemize}
Therefore, the proof of Theorem \ref{thm:intro} follows from the hypotheses.

\section{Proof of Theorem \ref{teob}}\label{sec:teob}
Consider $x(T,z,\e)$ the solution of the differential equation \eqref{s1} written as 
\[
x(T,z,\e)=x(T,z,0)+\sum_{i=1}^{k}\e^{i}\dfrac{y_i(T,z)}{i!}+\CO(\e^{k+1}).
\]
The Poincar\'e map associated with the differential equation \eqref{s1}, $\Pi(z,\e)=x(T,z,\e)$ becomes 
\begin{equation*}
\Pi(z,\e)=x(T,z,0)+h(T,z,\e),
\end{equation*}
where $h(T,z,\e)$ denotes the function $h(T,z,\e)=\sum_{i=1}^{k}\e^{i}\dfrac{y_i(T,z)}{i!}+\CO(\e^{k+1})$. 

From Hypothesis {\bf H3} we have that $\varphi(0,\e)$ is a fixed point of the Poincar\'e map, then the stability of the periodic solution $\varphi(t,\e)$ is provided by the eigenvalues of the Jacobian matrix
\begin{equation}\label{eqc}
\partial_z\Pi(\varphi(0,\e),\e)=Y_\e(\varphi(0,\e))+\partial_zh(T,\varphi(0,\e),\e).
\end{equation}
In order to study the eigenvalues of $\partial_z\Pi(\varphi(0,\e),\e)$ we consider the similar matrix
\begin{equation}\label{mp0}
L\,\partial_z \Pi(\varphi(0,\e),\e)\,L^{-1}=L\, Y_0(T)\,L^{-1}+L (Y_\e(T)-Y_0(T)+\partial_zh(T,\varphi(0,\e),\e))L^{-1},
\end{equation}
where $L$ is given in \eqref{ML}. Accordingly, from \eqref{LYIL} we obtain 
\begin{align*}
L\,\partial_z \Pi(\varphi(0,\e),\e)\,L^{-1}=\left(\begin{array}{cc}
I_m & 0\\
0& I_{n-m}+\Delta(\al^*)-\p\beta (\al^*)\Gamma(\al^*)\end{array}\right)+\left(\begin{array}{cc}A(\e)&B(\e)\\C(\e)&D(\e)\end{array}\right),
\end{align*}
with  $A(\e)$, $B(\e)$, $C(\e)$ and $D(\e)$  defined  in \eqref{mp1}. In order to study the eigenvalues of the above matrix we use Schur's determinant formulae \cite[Pg. 46]{schur} for matrices of the form
$$
\mathcal{M}=\left(\begin{array}{cc}
\mathcal{M}_1 & \mathcal{M}_2\\
\mathcal{M}_3& \mathcal{M}_4\end{array}\right),
$$
where the sub-matrices $\mathcal{M}_1$,  $\mathcal{M}_2$, $\mathcal{M}_3$ and $\mathcal{M}_4$ has the following sizes ${m\times m}$, ${m\times m-n}$, ${m-n\times m}$ and ${n-m \times n-m}$, respectively. If either $\mathcal{M}_1$ or $\mathcal{M}_4$ is invertible, the determinant of $\mathcal{M}$ is  given by 
\begin{equation}\label{schur}
\begin{split}
\det(\mathcal{M})=&\det(\mathcal{M}_1)\det(\mathcal{M}_4-\mathcal{M}_3\mathcal{M}_1^{-1}\mathcal{M}_2),\,\, \mbox{or}\\
\det(\mathcal{M})=&\det(\mathcal{M}_4)\det(\mathcal{M}_1-\mathcal{M}_2(\mathcal{M}_4)^{-1}\mathcal{M}_3),
\end{split}
\end{equation}
respectively. 

Since the matrices $\partial_z\Pi(\varphi(0,\e),\e)$ and $L\,\partial_z \Pi(\varphi(0,\e),\e)\,L^{-1}$ share the same eigenvalues, we shall use the Schur's formulae \eqref{schur} to study the roots of the characteristic polynomial $\det[L\,\partial_z \Pi(\varphi(0,\e),\e)\,L^{-1}-\omega I_n]=$
\begin{equation}\label{detsim}
\begin{aligned}
\det\left[
\left(\begin{array}{cc}
(1-\omega) I_m +A(\e)& B(\e)\\
C(\e)& (1-\omega)I_{n-m}+\Delta(\al^*)-\p\beta (\al^*)\Gamma(\al^*)+D(\e)\end{array}\right)\right].
\end{aligned}
\end{equation}
From \eqref{detsim}, we see that there are two group of eigenvalues for \eqref{mp0}, namely, the ones related with the block 
$$
\mathcal{M}_1(\e)=(1-\omega) I_m +A(\e)
$$
and the ones related with the block 
$$
\mathcal{M}_4(\e)=(1-\omega) I_{n-m}+\Delta(\al^*)-\p\beta (\al^*)\Gamma(\al^*)+D(\e).
$$
The eigenvalues related to the block $\mathcal{M}_1(\e)$ can be studied as follows:

Taking $K_1$ a compact region in the complex plane which does not contain $1$, we have that $\det(I_m-\omega I_m )\neq 0$ for $\omega\in K_1$ and, consequently, the submatrix $\mathcal{M}_1(\e)$ is invertible for every $\omega\in K_1$ and $\epsilon>0$ sufficiently small. Hence, from \eqref{schur}, we have that
\begin{equation}\label{pol1}
\det(L\,\partial_z \Pi(\varphi(0,\e),\e)\,L^{-1}-\omega I_n)=\det\left((1-\omega) I_m +A(\e)\right)\CP(\omega;\e)
\end{equation}
where $\CP(\omega;\e)$ is given by \eqref{cor}. We are going to show that the zeroes of $\CP(\omega;\e)$ provide $n-m$ eigenvalues of \eqref{mp0}.

The remaining $m$ eigenvalues, related to the block $\mathcal{M}_4(\e)$, can be studied as follows. Taking $\omega=1+\e^\ell \lambda$, one has 
\[
\mathcal{M}_4(\e)=\Delta(\al^*)-\p\beta (\al^*)\Gamma(\al^*)-\e^{\ell}\la I_{n-m}+D(\e).
\]
Thus, by hypothesis {\bf H4} we have that $\det(\mathcal{M}_4(0))\neq0$. Consequently, there is a compact region in the complex plane $K_2$, such that $\mathcal{M}_4(\e)$ is invertible for every $\la\in K_2$ and $\e>0$ sufficiently small. Hence, for $\la\in\mathbb{C}$, we have
\begin{equation}\label{pol2}
\det\left(L\,\partial_z \Pi(\varphi(0,\e),\e)\,L^{-1}-(1+\e^\ell \lambda) I_n\right)=\det\left(N(\e)-(1+\e^\ell\lambda)I_{n-m}\right)\CQ(\la;\e),
\end{equation}
where $\CQ(\la;\e)$ is given by \eqref{est}.
Namely, since $M(\e) $ defined in \eqref{NM} satisfies \eqref{Mep}, we have from \eqref{schur} that
\begin{equation*}
\begin{split}
\CQ(\la;\e)&=\det\Big[-\e^\ell \lambda I_{m}+A(\e)-B(\e)\mathcal{M}_4(\e)^{-1}C(\e)\Big]\\
&=\e^{m\ell}\det\left(M_\ell-\lambda I_{m} \right)+\e^{m\ell+1}R(\lambda,\e).
\end{split}
\end{equation*}
The roots of $\CQ(\la;\e)$ will provide the other $m$ eigenvalues of \eqref{mp0}. Therefore, w denote the characteristic polynomial of \eqref{eqc} as
$$
\mathcal{C}(\omega,\e)=\det\left(L\,\partial_z P(\varphi(0,\e),\e)\,L^{-1}-\omega I_n\right).
$$
Taking \eqref{detsim} into account, we have that
$$
\mathcal{C}(\omega,\e)=\mathcal{C}_0(\omega)+\e \mathcal{R}(\omega,\e),
$$
where 
$$
\mathcal{C}_0(\omega)=\det\left(\begin{array}{cc}
 (1-\omega) I_m& 0\\
0& (1-\omega)I_{n-m}+\Delta(\al^*)-\p\beta (\al^*)\Gamma(\al^*)\end{array}\right).
$$
Let $n^s$, $n^u$, and $n^c$ be the number of roots of $\mathcal{C}_0(\omega)$ with norm smaller than one, greater than one, and equal to one, respectively. Following the same approach as in the previous proof, we consider two simple curves $C^u\subset \{z\in\mathbb{C}:\,|z|>1\}$ and $C^s\subset \{z\in\mathbb{C}:\,|z|<1\}$ such that $\inte(C^u)$ contains  all the $n^u$ roots of $\mathcal{C}_0(\omega)$ that are outside the unitary disc, and   $\inte(C^s)$ contains all the $n^s$ roots of $\mathcal{C}_0(\omega)$ within the unitary disc, see Figure \ref{fig1}. There exists $\ov \e>0$ such that  $|\mathcal{C}_0(z)|>|\mathcal{C}(z,\e)-\mathcal{C}_0(z)|$  for every $z\in C^u\cup C^s$. Applying Rouche's Theorem we have that, for every $\e\in[0,\ov \e]$,  polynomials $\mathcal{C}_0(\omega)$ and $\mathcal{C}(\omega,\e)$ have the same number of roots inside $\inte(C^u)$ and $\inte(C^s)$, counting multiplicity. 

In order to study the branches of roots of  $\mathcal{C}(\omega,\e)$ that emerge from the $n^c$ roots of $\mathcal{C}_0(\omega)$ that are on the unitary disc (couting multiplicity), we divide them in two sets. First,  we consider  a closed ball $C_a^c $ around $1$, which is a root with multiplicity $m$ of polynomial $\mathcal{C}_0(\omega)$. Then, we define a region bounded by a simple curve $C^c_b$ that contains $n^c-m$ roots of $\mathcal{C}_0(\omega)$ on the unitary disc satisfying $C_b^c \cap C_a^c=\emptyset$, $C_b^c \cap C^u=\emptyset$, and $C_b^c \cap C^s=\emptyset$.  Figure \ref{fig1}, provides the qualitative configuration of the regions $C^u$, $C^s$, $C_a^c$, and $C_b^c$.
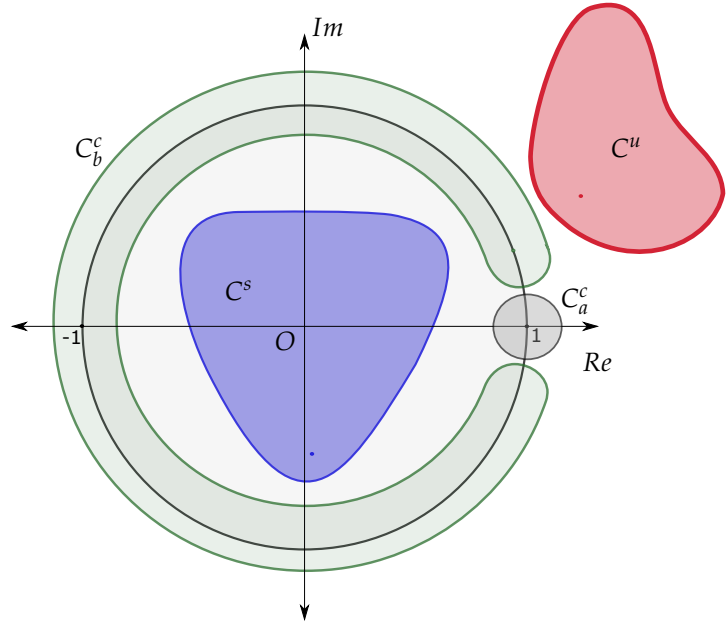
\begin{figure}[H]
	\begin{overpic}[width=9.5cm]{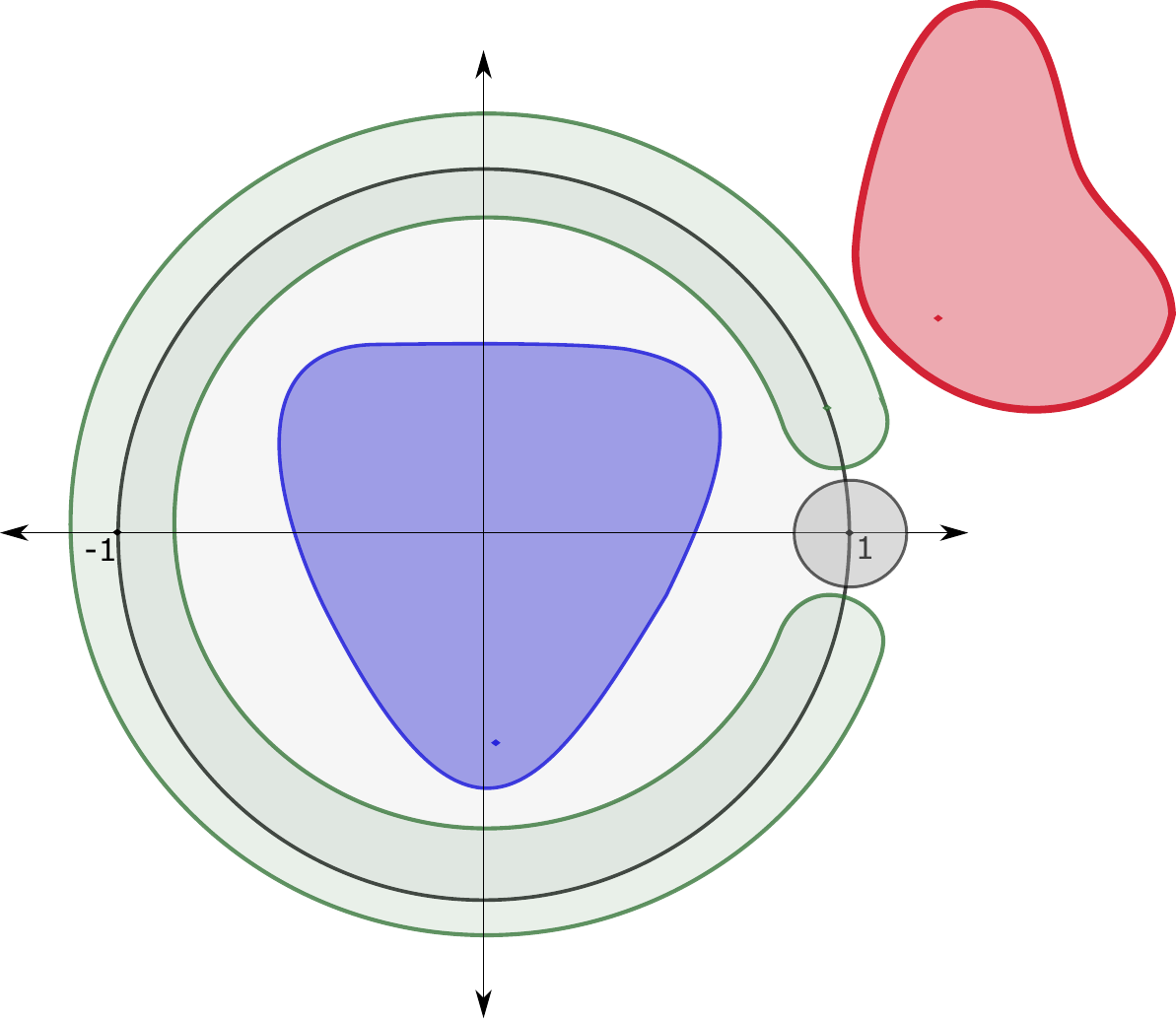}
\put(84,65){$C^u$}	
\put(30,45){$C^s$}
\put(9,65){$C_b^c$}	
\put(77,44){$C_{a}^c$}	
\put(37,38){$O$}	
\put(42,82){$Im$}
\put(80,35){$Re$}
	\end{overpic}
	\caption{Configuration of the bounded regions containing the roots of $\mathcal{C}(\omega,\e)$.}
	\label{fig1}	
\end{figure}

Henceforth, we shall establish the connection between the roots inside $C_a^c$ and $C_b^c$ with the zeroes of the functions $\CP(\omega;\e)$ and $\CQ(\la;\e),$ respectively.  First, we restrict our domain to $K_2=\ov{\textrm{Int}C_a^c}$, for $\e>0$ sufficiently small the sub-matrix $\mathcal{M}_4(\e)= (1-\omega)I_{n-m}+\Delta(\al^*)-\p\beta (\al^*)\Gamma(\al^*)+D(\e)$ is invertible. From \eqref{pol2} we have the following relationship
\[
\mathcal{C}(1+\e^\ell \lambda,\e)=0 \iff \mathcal{Q}(\lambda,\e)=0.
\]
In order to compare the roots of $\mathcal{Q}(\lambda,\e)$ with the roots of its  $\mu_2+m\ell-$jet we write
\begin{align*}
\mathcal{Q}_{\mu_2}(\lambda,\e)&=J_{\mu_2+m\, \ell}^{\e}\CQ(\lambda,\e)\\
&=\e^{m\ell}\det\big(M_{\ell}-\lambda I_{m}\big)+\e^{m\ell+1}\mathcal{R}_{\mu_2}(\lambda,\e).
\end{align*}
This is analogous to what was done in the proof of Theorem \ref{teoa}. In fact, for $\e>0$ sufficiently small, $|\omega(\e)|=|1+\e^\ell \lambda(\e)|<1$ (resp. $|\omega(\e)|=|1+\e^\ell \lambda(\e)|>1$)
 provide that $\Re(\lambda(0))<0$ (resp. $\Re(\lambda(0))>0$). Thus, it follows from Rouche's Theorem  that $\mathcal{Q}(\lambda,\e)\e^{-m\ell}$ and $\mathcal{Q}_{\mu_2}(\lambda,\e)\e^{-m\ell}$ have the same number of zeroes counting multiplicity. Moreover, inside the region $C_a^c$ we can define smaller regions bounded by simple curves $D^u$, $D^s$ and $D^c$ that  contain, for $\e\in[0,\ov\e]$, the roots $1+\e^\ell \lambda(\e)$ satisfying $\Re(\lambda(0))>0$, $\Re(\lambda(0))<0$ and $\Re(\lambda(0))=0$ respectively, see Figure \ref{fig2}.
\begin{figure}[H]
	\begin{overpic}[width=7cm]{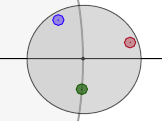}
\put(81,41){$D^u$}	
\put(40,58){$D^s$}	
\put(54,15){$D^c$}
\put(52.5,33.5){$1$}
\put(80,10){$C_{a}^c$}
	\end{overpic}
	\caption{The smaller regions $D^s$,$D^u$, and $D^c$ inside $C_a^c$.}
	\label{fig2}
\end{figure}
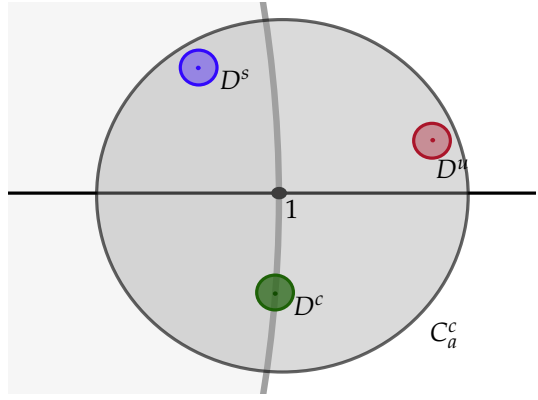
Thus, concerning the roots that belongs to $D^c$ we observe that 
\begin{align*}
\widehat{\mathcal{Q}}(\lambda,\e)=\mathcal{Q}(\lambda,\e)\e^{-m\ell}&=\det\big(M_{\ell}-\lambda I_{m}\big)+\e\mathcal{R}(\lambda,\e)\\
\widehat{\mathcal{Q}}_{\mu_2},(\lambda,\e)=\mathcal{Q}_{\mu_2}(\lambda,\e)\e^{-m\ell}&=\det\big(M_{\ell}-\lambda I_{m}\big)+\e\mathcal{R}_{\mu_2}(\lambda,\e)  ,
\end{align*}
then we have $\widehat{\mathcal{Q}}(\lambda,0)=\widehat{\mathcal{Q}}_{\mu_2}(\lambda,0)=\det\big(M_{\ell}-\lambda I_{m}\big)$. Taking $\lambda(0)=\lambda^*$ satisfying $\Re(\lambda^*)=0$ and $\det\big(M_{\ell}-\lambda^* I_{m}\big)=0$ we have by hypothesis that $\lambda^*$ is a simple root of the characteristic polynomial of $M_\ell$. Consequently, we have $\partial_\lambda\widehat{\mathcal{Q}}(\lambda^*,0)\neq 0$ and  $\partial_\lambda\widehat{\mathcal{Q}}_{\mu_2}(\lambda^*,0)\neq 0$. Then, by the Implicit Function Theorem, there exists a unique root $\lambda(\e)$ of $\widehat{\mathcal{Q}}(\lambda,\e)$ and a unique  root $\lambda_{\mu_2}(\e)$ of $\widehat{\mathcal{Q}}_{\mu_2}(\lambda,\e)$, satisfying $\lambda(0)=\lambda_{\mu_2}(0)=\lambda^*$. The branch of zeros $\lambda(\e)$ and $\lambda_{\mu_2}(\e)$ are $\mathcal{C}^k$ in a neighbourhood of $\e=0$. Moreover,  $J_{\mu_2}\lambda(\e)=J_{\mu_2}\lambda_{\mu_2}(\e)$ and, consequently,
\[
\sgn\left(\Re(\lambda(\e))\right)=\sgn\left(\Re(J_{\mu_2}\lambda(\e))\right)=\sgn\left(\Re(J_{\mu_2}\lambda_{\mu_2}(\e))\right),
\]
for $\e>0$ sufficiently small, provided that $\Re(J_{\mu_2}\lambda_{\mu_2}(\e))\neq 0$. Thus, it follows that 
\begin{align*}
|\omega(\e)|^2&=1+2\e^\ell \Re\left(\lambda(\e)\right)+\e^{2\ell}|\lambda(\e)|^2
\end{align*}
then we have 
\begin{align*}
\sgn\left(|\omega(\e)|^2-1\right)&=\sgn\left(2\e^\ell \Re\left(\lambda(\e)\right)+\e^{2\ell}|\lambda(\e)|^2\right)\\
&=\sgn\left(\dfrac{2\e^\ell \Re(J_{\mu_2}\lambda_{\mu_2}(\e))+\e^{2\ell}|\lambda(\e)|^2}{2\e^\ell \Re(J_{\mu_2}\lambda_{\mu_2}(\e)}\left(2\e^\ell \Re(J_{\mu_2}\lambda_{\mu_2}(\e)\right)\right)\\
&=\sgn\left(1+\CO(\e)\right)\sgn\left(\Re(J_{\mu_2}\lambda_{\mu_2}(\e))\right)\\
&=\sgn\left(\Re(J_{\mu_2}\lambda_{\mu_2}(\e))\right).
\end{align*}
Finally, we restrict our domain to the region $K_1=\ov{\textrm{Int}\,C_b^c}$. Since $1\not\in K_1$,  we have for $\e$ sufficiently small, the sub-matrix $\mathcal{M}_1(\e)=I_m-\omega I_m +A(\e)$ is invertible.

On the other hand, from \eqref{pol1} we have the following relationship
\begin{equation*}
\mathcal{C}(\omega,\e)=0 \iff \mathcal{P}(\omega;\e)=0.
\end{equation*}
Therefore, the only zeros $\omega(\e)$ of 
\begin{equation*}
\begin{split}
\CP(\omega;\e)&=\det\Big[N(\e)-\omega\, I_{n-m}-C(\e)\big((1-\omega)I_m+\e\, A(\e)\big)^{-1} B(\e)\Big]\\
&=\det\left( (1-\omega)I_{n-m}+\Delta(\al^*)-\p\beta (\al^*)\Gamma(\al^*)\right)+\e \mathcal{R}(\omega,\e),
\end{split}
\end{equation*} 
within $C_b^c$ are those satisfying $|\omega(0)|= 1$. Thus, taking $\omega(0)=\omega^*$ satisfying $|\omega^*|=1$, we have by hypothesis that $\omega^*$ is a simple root of the characteristic polynomial $\det( (1-\omega)I_{n-m}+\Delta(\al^*)-\p\beta (\al^*)\Gamma(\al^*))$. Consequently, $\partial_\omega \CP(\omega^*;0)\neq 0$ and it follows from the Implicit Function Theorem that there exists  a unique root $\omega(\e)$ of $\CP(\omega;\e)$ for $\e$ sufficiently small. Following the same argument as before, we can check that if $\omega_{\mu_1}(\e)$ is the unique zero of the  jet $J_{\mu_1}^{\e} \CP(\omega;\e)$ satisfying $\omega_{\mu_1}(0)=\omega^*$, then $J_{\mu_1}\omega(\e)=J_{\mu_1}\omega_{\mu_1}(\e)$ for  $\e>0$ sufficiently small. The norm of $\omega(\e)$ can be determined by the following computation
\begin{align*}
|\omega(\e)|^2&=|J_{\mu_1}\omega(\e)|^2+\e^{\mu_1+1}\left(\ov{\omega(\e)}\,J_{\mu_1} \omega(\e)-2\,|J_{\mu_1} \omega(\e)|^2+\omega(\e)\,\ov{J_{\mu_1} \omega(\e)}\right)\\
&+\e^{2(\mu_1+1)}\left(|J_{\mu_1} \omega(\e)|^2-\ov{\omega(\e)}\,J_{\mu_1} \omega(\e)-{\omega(\e)}\,\ov{J_{\mu_1} \omega(\e)}+|\omega(\e)|^2 \right).
\end{align*}
Thus, provided that $|J_{\mu_1}\omega_{\mu_1}(\e)|^2-1 \neq 0$ and $\e>0$ sufficiently small, we have
\begin{align*}
\sgn\left(|\omega(\e)|^2-1\right)&=\sgn\left(|J_{\mu_1}\omega_{\mu_1}(\e)|^2-1+\CO(\e^{\mu_1+1})\right)\\
&=\sgn\left(\dfrac{|J_{\mu_1}\omega_{\mu_1}(\e)|^2-1+\CO(\e^{\mu_1+1})}{|J_{\mu_1}\omega_{\mu_1}(\e)|^2-1}\left(|J_{\mu_1}\omega_{\mu_1}(\e)|^2-1\right)\right)\\
&=\sgn\left(1+\CO(\e)\right)\,\sgn\left(|J_{\mu_1}\omega_{\mu_1}(\e)|^2-1\right)\\
&=\sgn\left(|J_{\mu_1}\omega_{\mu_1}(\e)|^2-1\right).
\end{align*}
Organizing the  results obtained so far, we can summarize our proof in the following way:
\begin{itemize}
\item[-] Let $\omega(\e)$ be one of the $n^u$ (resp. $n^s$) zeroes of $\CP(\omega;\e)$ satisfying $|\omega(0)|>1$ (resp. $|\omega(0)|<1$). Then, $\omega(\e)$ is on of the roots of the characteristic polynomial
$$
\mathcal{C}(\omega,\e)=\det\left(\mathcal{M}_1\right)\CP(\omega;\e),
$$
for $\e>0$ sufficiently small.

\item[-] If  $\omega(\e)$  is  one of the $n^c-m$ zeroes of $\CP(\omega;\e)$ satisfying $|\omega(0)|=1$, then there exists $\widetilde\omega(\e)$ a zero of  $J^\e_{\mu_1}\CP(\omega;\e)$, with the same multiplicity of $\omega(\e)$ satisfying  $\widetilde\omega_{\mu_1}(\e)=\omega_{\mu_1}(\e)$ and
$$
\sgn\left(|\omega(\e)|^2-1\right)=\sgn\left(|J_{\mu_1}\widetilde\omega_{\mu_1}(\e)|^2-1\right),
$$
for $\e>0$ sufficiently small.

\item[-] If $\lambda(\e)$ is one of the $m$ zeroes of $\mathcal{Q}(\lambda,\e)$, counting multiplicity, then there exists a $\omega(\e)=1+\e^\ell \lambda(\e)$ a root of the characteristic polynomial 
$$
\mathcal{C}(\omega,\e)=\det\left(\mathcal{M}_4\right)\mathcal{Q}(\lambda;\e),
$$
where $\omega=1+\e^l \lambda$. The root $\omega(\e)$ satisfies the following relationships. It has the same multiplicity of  $\lambda(\e)$. Moreover, there exists a zero $\widetilde\lambda(\e)$ of $J^\e_{\mu_2+m\ell}\mathcal{Q}(\lambda,\e)$ such that $\widetilde\lambda_{\mu_2}(\e)=\lambda_{\mu_2}(\e)$ and 
$$
\sgn\left(|\omega(\e)|^2-1\right)=\sgn\left(\Re(J_{\mu_2}\widetilde\lambda_{\mu_2}(\e))\right),
$$ 
for $\e>0$ sufficient small.
\end{itemize}
This concludes the proof.

\section{General formulae for asymptotic expansions}\label{Generic}

The Lyapunov-Schmidt reduction is one of the methods used for detecting isolated periodic solutions bifurcating from a submanifold of $T-$periodic solutions. This type of periodic solution satisfies hypotheses {\bf H1}-{\bf H3}. 
In \cite{CLN17}, the authors combined the averaging method and the Lyapunov-Schmidt reduction for studying this kind of bifurcation problem for differential equations written in the standard form \eqref{s1}. As usual, in the averaging theory the problem is reduced to the analysis of certain bifurcation functions. Namely, from  \eqref{smoothyi} we define the \textit{averaged function of order $i$} by taking
\begin{equation}\label{eqgi}
\bg_i(z)=\dfrac{y_i(T,z)}{i!}, \quad i=1,\dots,k.
\end{equation}
For the sake of simplicity, we denote $y_0(T,z)=x(T,z,0)$ and
$$
\bg_0(z)=y_0(T,z)-z=x(T,z,0)-z.
$$ 
Using the Poincar\'e map \eqref{pm},  we can define the displacement function
\begin{equation}\label{fdm}
\begin{split}
d(z,\e)&=\Pi(z,\e)-z\\
&=\bg_0(z)+\sum_{i=1}^{k}\e^{i} \bg_i(z)+\CO(\e^{k+1}).
\end{split}
\end{equation}

Clearly, if $\varphi(t,\e)$ is a solution of the differential equation \eqref{s1} satisfying $d(\varphi(0,\e),\e)=0,$ then $\varphi(0,\e)$ is a fixed point of the Poincar\'e map $\Pi(z,\e)$. For the sake of completeness, in what follows we provide some results from \cite{CLN17}. In order to state \cite[Theorem A]{CLN17} we need to determine the bifurcation functions related with the displacement function.

Let  $\pi:\R^m\times\R^{n-m} \rightarrow\R^m$ and
$\pi^{\perp}:\R^m\times \R^{n-m} \rightarrow\R^{n-m}$  denote the
projections onto the first $m$ coordinates and onto the last $n-m$
coordinates, respectively. For a point $z\in\Omega,$ we also consider
$z=(a,b)\in \R^m\times\R^{n-m}$. The
bifurcation functions corresponding to the equation $d(z,\e)=0$ are
\begin{equation}\label{fi} f_i(\al)=\pi
\bg_{i}(z_{\al})+\sum_{l=1}^i\sum_{S_l}\dfrac{1}{c_1!\,c_2!2!^{c_2}\cdots
c_l!l!^{c_l}}\p_b^L\pi
\bg_{i-l}(z_{\al})\bigodot_{j=1}^l\ga_j(\al)^{c_j},
\end{equation}
\begin{equation*}
\CF^{k}(\al,\e)=\sum_{i=1}^{k} \e^i f_i(\al),
\end{equation*}
where $\ga_i:V\rightarrow \R^{n-m}$, for $i=1,2,\ldots,k$, are
defined recurrently as
\begin{align*}
\ga_1(\al)=\!&-\Delta_{\al}^{-1}\pi^{\perp}\bg_{1}(z_\al)
\quad\text{and}\vspace{0.3cm}\nonumber\\
\ga_i(\al)=\!&-i!\Delta_{\al}^{-1}\Bigg(
\sum_{S'_i}\dfrac{1}{c_1!\,c_2!2!^{c_2}\cdots
c_{i-1}!(i-1)!^{c_{i-1}}}\p_b^{I'}\pi^{\perp}
\bg_{s}(z_{\al})\bigodot_{j=1}^{i-1}
\ga_j(\al)^{c_j}\vspace{0.2cm}\\
&+ \sum_{l=1}^{i-1}\sum_{S_l}\dfrac{1}{c_1!\,c_2!2!^{c_2}\cdots
c_l!l!^{c_l}}\p_b^L\pi^{\perp}
\bg_{i-l}(z_{\al})\bigodot_{j=1}^l\ga_j(\al)^{c_j}\Bigg),\nonumber
\end{align*}
with  $\Delta_\al=\p\pi^{\perp} \bg_0(z_{\al})$.

\begin{theorem}\label{PStf}
Assume that system \eqref{s1} satisfies  hypothesis  {\bf H1} and {\bf H2}.
Consider the Jacobian matrix
$$
\p \bg_0({\bf z}_\alpha)=\begin{pmatrix}
* & \Gamma_\alpha \\
* & \Delta_\alpha
\end{pmatrix},
$$
where $\Gamma_\alpha=\p_b\pi
\bg_0(z_\alpha)$ and
$\Delta_\alpha=\p_b\pi^\perp \bg_0(z_\alpha)$. In addition to these hypothesis  we suppose that
\begin{itemize}

\item[$(i)$] for some $r\in\{0,\ldots,k\}$,  $f_1 = f_2 =\dots = f_{r-1}
= 0$ and $f_r$ is not identically zero;

\item[$(ii)$] there exists $\ov \e>0$  such that for each
$\e\in(-\ov\e,\ov\e)$ there exists $a_{\e}\in V$ satisfying
$\CF^{k}(a_{\e},\e )=0$; and

\item[$(iii)$] there exist a constant $P_0>0$ and a positive integer $l\leq
( k+r+1)/2$ such that
\begin{equation*}
\left|\p_\al\CF^{k}(a_\e,\e)\cdot \al\right|\geq P_0|\e^l||\al|,
\quad \text{for} \quad \al\in V.
\end{equation*}
\end{itemize}
Then, for $|\e|\neq0$ sufficiently small there exists a
$T$-periodic solution $\varphi(t,\e)$ of the differential equation \eqref{s1} such that
$|\pi\,\varphi(0,\e)-\pi\,z_{a_{\e}}|=\CO(\e^{ {k}+1-l})$, and
$|\pi^{\perp}\varphi(0,\e)-\pi^{\perp}z_{a_{\e}}|=\CO(\e)$.
\end{theorem}
If the additional hypothesis that $f_r$ has a simple root is assumed, Theorem \ref{PStf} can be written in a simpler form that resemble the classical Theorem \ref{thm:intro}.
\begin{corollary}\label{1popct1a}
In addition to hypotheses {\bf H1} and {\bf H2} we assume that  $r\in\{1,\ldots,k\}$ is the first sub-index such that $f_r$ is  not identically zero. If there exists $\al^* \in V$ such that
$f_r(\al^*)=0$ and  $\det\left(\partial f_r (\al^*) \right)\neq
0$, then there exists a $T$-periodic solution $\varphi(t,\e)$ of
\eqref{s1} such that  $|\varphi(0,\e)-z_{\al^*}|=\CO(\e)$.
\end{corollary}

Theorem \ref{PStf} is still true when $m=n$. In this case, we assume that
$V$ is an open subset of $\R^n$ then $\CZ =\ov V \subset \Omega$ and the
projections $\pi$ and $\pi^\perp$ become the identity and the null
operator, respectively. Moreover, the bifurcation functions
$f_i:V\rightarrow \R^n$, for $i=1,2,\dots,k$, become the averaged functions $f_i(\al)=\bg_i(\al)$ defined in \eqref{eqgi}. 

Consider $m=n$,  $z_\al=\al\in\CZ$ and the hypothesis {\bf H1.}  Thus, the result of Theorem  \ref{PStf} holds without any assumption about $\Delta_\al$. Thus, we have the following corollary.

\begin{corollary}\label{cte}
Assume that $\bg_i= 0$ for all $i\in \{ 0,\dots,k-1\}$. If
there exists ${z}^*\in \Omega$ such that  $\bg_{k}({z}^*)=0$ and
$\partial \bg_{k}({z}^*)\neq 0$, then there exists a $T$-periodic solution
$\f(t,\e)=x(t,{z}(\e),\e)$ for the differential equation \eqref{s1} such that
$\f(0,\e)={z}^*+\CO(\e)$.
\end{corollary}

\begin{remark}\label{rmk}
In the next section we shall study the stability of the periodic solutions provided by Theorem \ref{PStf}. In order to have all information needed we present here  a sketch of the proof of Theorem \ref{PStf}. For a detailed proof see \cite[Section $3$]{CLN17}. The proof starts by using the projections $\pi$ and $\pi^\perp$ for transforming the equation $d(z,\e)=0$ into the equivalent system of equations
\begin{equation*}
\left\{\begin{array}{r}
\pi d(a,b,\e)=0,\vspace{0.1cm}\\
\pi^{\perp}d(a,b,\e)=0,
\end{array}\right.
\end{equation*}
and, then, dealing with the lines, separately.  Accordingly, hypothesis {\bf H2} provides the existence of a $\C^{k+1}$
function $\ov\be:U\times(-\e_2,\e_2)\rightarrow\R^{n-m}$, defined on a neighbourhood $U\times(-\e_2,\e_2)$ of $\ov V\times\{0\}$, satisfying $\pi^{\perp}d(a,\ov{\be}(a,\e),\e)=0$ for
each $(a,\e)\in U\times(-\e_2,\e_2)$ and $\ov{\be}(\al,0)=\be(\al)$
for all $\al\in\ov V$. The function $\ov\be$ can be expanded in Taylor series around $\e=0$ as
\begin{equation}\label{b0}
\beta (\al,\e)=\beta(\al)+\e
\ga_1(\al)+\cdots+\ga_{k}(\al)+\CO(\e^{{k}+1}),
\end{equation}
where $\ga_i$, for $1\leq i\leq {k}$, is defined in \eqref{fi}.  
 Now, hypotheses $(i)$ and $(ii)$  provide the
existence of $\e_2>0$ and a function $\al:(-\e_2,\e_2)\rightarrow \R^m$ satisfying $\pi d(\al(\e),{\be}(\al(\e),\e),\e)=0$ for all $\e\in(-\e_2,\e_2)$. Moreover
\begin{equation*}
\al(\e)=a_\e+\CO(\e^{{k}-l+1}).
\end{equation*}
Consequently, taking $z(\e)=\left( \al(\e),{\be}(\al(\e),\e)\right)$, we conclude that
$d(z(\e),\e)=0$ for all $\e\in(-\e_2,\e_2)$. Hence, for $|\e|\neq0$ sufficiently small, $\varphi(t,\e)=x(t,z(\e),\e)$ is a $T-$periodic solution of the differential equation \eqref{s1}. 

We are going to use the information provided by the functions $\al(\e)$, $\gamma_i(\al)$ and  $f_i(\al)$ in order to study the stability of the periodic solution $\varphi(t,\e)=x(t,z(\e),\e)$.
\end{remark}
Under the hypothesis of Corollary \ref{1popct1a}  we can use the Fa\'a di Bruno formula and the implicit function theorem for obtaining the coefficients of the expansion in power series of the function $a_{\e}$ present in statement $(ii)$ of Theorem \ref{PStf}. These expressions are required for obtaining the formulae to be developed in the next section.
\begin{proposition} \label{c1}  Consider $\Delta_\al$, $r$, and $\al^*\in V$ satisfying the hypothesis of Corollary \ref{1popct1a}. Thus, there exists $a_\e$ such that the following statements hold:
\begin{itemize}
\item[$(a)$]For $\e>0$ sufficiently small, there exists a $\C^k$ function $a_\e$ such that $\CF^{k}(a_\e,\e)=0$,
\item[$(b)$]$a_\e =\al_0+\e \al_1+\cdots+\e^{k}\al_{k}+\CO(\e^{k+1})$
with $\al_i \in \R^n$ for all $1\leq i \leq k$,  where the coefficients are given by
\[
\begin{aligned}
\al_0&=\al^*,\\
\al_{1}&=-\partial  f_r
(\al^*)^{-1}f_{r+1}(\al^*),\\
\al_i&=-\dfrac{1}{i!} \partial  f_r(\al^*)^{-1}\Bigg(
\sum_{S'_i}\dfrac{1}{c_1!\,c_2!2!^{c_2}\cdots
c_{i-1}!(i-1)!^{c_{i-1}}}\p^{I'}f_r(\al^*)\bigodot_{j=1}^{i-1}
\left(\p^{j}\al(0)\right)^{c_j}\vspace{0.2cm}\\
&+ \sum_{l=1}^{i-1}\sum_{S_l}\dfrac{1}{c_1!\,c_2!2!^{c_2}\cdots
c_l!l!^{c_l}}\p^{L}f_{i-l+r}(\al^*)\bigodot_{j=1}^l\left(\p^{j}\al(0)\right)^{c_j}
\Bigg), \,\,\, 2\leq i\leq k-1.
\end{aligned}
\]
\end{itemize}
\end{proposition}
 
The next lemma combine the implicit function theorem and the Fa\'a di Bruno formulae, and it will be applied for proving Proposition \ref{c1}.

\begin{lemma}\label{l1}
Let $u:\R^n\times [0,\e_0]\rightarrow \R^n$ be a function of class
$\C^k$ such that
\begin{equation*}\label{sumu}
u(x,\e)=u_0(x)+\e  u_1(x)+\dots+\e^k u_k(x)+\CO(\e^{k+1}).
\end{equation*}
Assume that there exists a function $v:\R \rightarrow \R^n$ of class
$\C^k$ satisfying $u(v(\e),\e)=0$ for $|\e|>0$ sufficiently small
and that the Jacobian matrix $\p u_0(v(0))$ is invertible. Then
$$
v(\e)=v(0)+\dfrac{\e}{1!}\p v(0)\,+\dots+\dfrac{\e^k}{k!}\p^{k}v
(0)+\CO(\e^{k+1}),
$$
where the Taylor's coefficients of $v(\e)$ are written recursively as
\linebreak $\p v(0)=\partial  u_0(v(0))^{-1}u_1(v(0))$ and for $2\leq i
\leq k$,
\begin{align*}
v^{i}(0)=&-\partial  u_0(v(0))^{-1}\Bigg(
\sum_{S'_i}\dfrac{1}{c_1!\,c_2!2!^{c_2}\cdots
c_{i-1}!(i-1)!^{c_{i-1}}}\p^{I'}u_0(v(0))\bigodot_{j=1}^{i-1}\left(
\p^{j}v(0)\right)^{c_j}\vspace{0.2cm}\nonumber\\
&+ \sum_{l=0}^{i-1}\sum_{S_l}\dfrac{1}{c_1!\,c_2!2!^{c_2}\cdots
c_l!l!^{c_l}}\p^{L}u_{i-l}(v(0))\bigodot_{j=1}^l\left(\p^{j}v(0)\right)^{c_j}\Bigg).
\end{align*}
\end{lemma}
\begin{proof}[Proof of Lemma \ref{l1}]
For $\e>$ sufficiently small we have by hypothesis that
$u(v(\e),\e)=0$, then  for $1\leq i \leq k$ we obtain
$$
\dfrac{d^i}{d
\e^i}u(v(\e),\e)=\sum^i_{j=0}\sum^i_{q=0}\binom{i}{q}(\e^{j})^
{(i-q)}\dfrac{d^q}{d \e^q}u_j(v(\e))+\CO(\e)=0.
$$
Taking $\e=0$, $l=i-j$ and using the Fa\'{a} di Bruno's Formula we
have
\begin{align*}
0=&\sum^i_{l=1}\dfrac{i!}{l!}\dfrac{d^l}{d \e^l}u_{i-l}(v(0))+i!u_i(v(0))\\
=&\dfrac{d^i}{d \e^i}u_0(v(0))+\sum^{i-1}_{l=1}\dfrac{i!}{l!}
\dfrac{d^l}{d \e^l}u_{i-l}(v(0))+i!u_i(v(0))\\
=&\Bigg(\sum_{S'_i}\dfrac{i!}{c_1!\,c_2!2!^{c_2}\cdots
c_{i-1}!(i-1)!^{c_{i-1}}}\p^{I'}u_0(v(0))\bigodot_{j=1}^{i-1}
\left(\p^{j}v(0)\right)^{c_j}\Bigg)\nonumber\\
&+i! \partial u_0(v(0))v^{i}(0)+
\Bigg(\sum_{l=0}^{i-1}\sum_{S_l}\dfrac{i!}{c_1!\,c_2!2!^{c_2}\cdots
c_l!l!^{c_l}}\p^{L}u_{i-l}(v(0))\bigodot_{j=1}^l\left(\p^{j}v(0)\right)^{c_j}\Bigg)\\
&+i!u_i(v(0)).
\end{align*}
Then, we isolate $v^{i}(0)$ obtaining
\begin{align*}
v^{i}(0)&=-\partial u_0(v(0))^{-1}\Bigg(\sum_{S'_i}\dfrac{i!}{c_1!\,c_2!2!^{c_2}\cdots
c_{i-1}!(i-1)!^{c_{i-1}}}\p^{I'}u_0(v(0))\bigodot_{j=1}^{i-1}
\left(\p^{j}v(0)\right)^{c_j}\\
&+\sum_{l=0}^{i-1}\sum_{S_l}\dfrac{i!}{c_1!\,c_2!2!^{c_2}\cdots
c_l!l!^{c_l}}\p^{L}u_{i-l}(v(0))\bigodot_{j=1}^l\left(\p^{j}v(0)\right)^{c_j}+u_i(v(0))\Bigg).
\end{align*}
\end{proof}
\begin{proof}[Proof of Proposition \ref{c1}]
Statement $(a)$ follows directly from the Implicit Function
Theorem. The  statement $(b)$ can be proved as follows. Define the function
$$
u(\al, \e)=\CF^{k}(\al,\e)=\sum_{i=1}^{k} \e^{i} f_i(\al),
$$
where $u_i(\al)=f_{i}(\al)$ for all $0<i<k$ are given in
\eqref{fi}. Then, from statement $(a)$ we have $u(a_\e, \e)=0$ and by
hypothesis we know that $\p u_0(\al_0)$ is invertible. Thus, we apply
Lemma \ref{l1} on $u(\al, \e)$ by taking $v(\e)=a_\e$, obtaining the
functions shown in statement $(b)$.
\end{proof}
\subsection{Asymptotic expansion of the Jacobian matrix of a periodic solution}
In the previous sections, a method for determine the stability of the periodic solution $\varphi(t,\e)$ defined in {\bf H3} was presented. The proposed approach involves the asymptotic expansion of the Jacobian matrix $\partial_z\Pi(\varphi(0,\e),\e)$ presented in \eqref{eqc}. The next result provide the formulae for the coefficients of this expansion.
 
\begin{proposition}\label{p6} Let  $a_\e$ be the function given in hypothesis $(ii)$ of Theorem $\ref{PStf}$ and
let $\varphi(t,\e)=x(t,\varphi(0,\e),\e)$  be the periodic solution of the differential equation \eqref{s1} provided by Theorem  \ref{PStf}. If
\begin{equation*}
a_\e=\al_0+\e \al_1+\cdots+\e^{k-l}\al_{k-l}+\CO(\e^{k-l+1}),
\end{equation*}
with $\al_i \in \R^m$ for all $0\leq i \leq k-l$. Then, the following statement hold:
\begin{itemize}
\item[(a)] the initial condition of the periodic orbit writes
\begin{equation}\label{z0}
\varphi(0,\e)=\sum_{i=0}^{k-l}\e^i\big(\al_i,\beta_i\big)+\CO\big(\e^{k-l+1}\big),
\end{equation}
where $\beta_0=\beta(\al_0)$ and for all $1\leq i \leq k-l$,
\begin{equation*}
\beta_i=\ga_i(\al_0)+\sum_{j=1}^i\sum_{S_j}\dfrac{1}{c_1!\,c_2!2!^{c_2}\cdots
c_j!j!^{c_j}}\p^{j}\ga_{i-j}(\al_0)\bigodot_{s=1}^j (s!\al_s)^{c_s}.
\end{equation*}
\item[(b)] the Jacobian matrix of the  Poincar\'e map  at $z=z(\e)$ can be written as
\begin{equation*}
\partial_z\Pi(\varphi(0,\e),\e)=A_0+\e  A_1+\dots+\e^{k-l} A_{k-l}+\e^{k-l+1} R(\e).
\end{equation*}
Where the matrices $A_j$ satisfies, for all $0\leq j\leq k-l$, the following relationship. We have $A_0=\partial_z y_0(T,z)=\partial_z\bx(T,z,0)$ and
\begin{equation*}
A_j=\sum^j_{i=0}\dfrac{1}{(j-i)!}\sum_{S_i}\dfrac{1}{b_1!\cdots
b_i!(i-1)!^{b_i}}\partial^{I+1}_zy_{j-i}(T,z_0)\bigodot^i_{u=1}
\left(u!z_u\right)^{b_u},
\end{equation*}
for $1\leq j\leq k-l$, with $z_i=(\al_i,\beta_i)$  given in \eqref{z0} and $l$ as in Theorem $\ref{PStf}$.
\end{itemize}
\end{proposition}
 
The remaining of this section is dedicated to the proof of this result.
\begin{proof}[Proof of Proposition \ref{p6}]
For proving statement $(a)$ we define
$$
g(z,\e)=\bg_0(z)+\e \bg_{1}(z)+\e^2 \bg_{2}(z)+\dots+\e^{k}\bg_k(z)+\CO(\e^{k+1}),
$$
by Theorem \ref{PStf} (see Remark \ref{rmk}) there are functions $\al(\e)$ and $\ov{\beta}(\al,\e)$ such that 
$$
z_{\al(\e)}=\left(\al(\e), \ov{\beta}(\al(\e),\e) \right),
$$ 
satisfies $g(z_{\al(\e)},\e)=0$ for  all $|\e|\neq $ sufficiently small. Furthermore, taking $\ov k = k$ we have  that  $\al(\e)=a_\e+\CO(\e^{\ov k -l+1})$ and 
\begin{equation}\label{b0}
\ov{\beta}(\al,\e)=\beta(\al)+\e\gamma_1(\al)+\cdots+\gamma_{\ov k }(\al)+\CO(\e^{\ov k +1}). 
\end{equation}
By hypothesis  $a_\e=\al_0+\e \al_1+\cdots+\e^{\ov k -l}\al_{\ov k -l}+\CO(\e^{\ov k -l+1})$ thus we can write
\begin{equation}\label{a1}
\al(\e)=\al_0+\e \al_1+\cdots+\e^{k-l}\al_{\ov k -l}+\CO(\e^{\ov k -l+1}).
\end{equation}
Substituting \eqref{a1} in \eqref{b0} and expanding the result in
power series of $\e$ around $\e=0$ we have
\begin{equation*}\label{bb}
\ov{\beta}(\al(\e),\e)=\beta_0+\e
\beta_1+\cdots+\e^{\ov k -l}\beta_{\ov k -l}+\CO\big(\e^{\ov k -l+1}\big),
\end{equation*}
where the coefficients up to order $\ov k-l$ can be obtained by the recursive formula $\beta_0=\ov{\beta}(\al(0),0)=\beta(\al_0)$ and for $0< i \leq
\ov k -l$
\begin{align*}
\beta_i=&\dfrac{1}{i!}\sum^i_{j=0}\sum^i_{q=0}\binom{i}{q}(\e^{j})^
{(i-q)}\dfrac{d^q}{d \e^q}\ga_j(\al(\e))+\CO(\e)\Big|_{\e=0}\\
=&\dfrac{1}{i!}\Bigg(i!\ga_i\big(\al(\e)\big)+\sum^i_{j=1}\binom{i}{i-j}
\dfrac{d^{i-j}}{d\e^{i-j}}\ga_j(\al(\e))+\CO(\e)\Bigg)\Big|_{\e=0}.
\end{align*}
Using  the  Fa\'{a} di Bruno's formula  and taking $l=i-j$ the above expression becomes
\begin{align*}
\beta_i=\ga_i\big(\al(0)\big)+\sum^i_{l=1}\sum_{S_l}\dfrac{1}
{c_1!\cdots
c_l!(l-1)!^{c_l}}\p^{L}\ga_{i-l}(\al(0))\bigodot^l_{s=i}\left(
\p^{s}\al(0)\right)^{c_s},
\end{align*}
where $\al(0)=\al_0$ and $\p^{s}\al(0)=s!\al_s$. This conclude the proof Proposition \ref{p6}$(a)$. 

Now we prove statement $(b)$, in this case the displacement function \eqref{fdm} writes
$$
d(z,\e)=\sum^{k}_{i=0}\e^{i}\dfrac{y_{i}(T,z)}{i!}+
\CO(\e^{k+1}).
$$
Thus the Poincar\'e map of of the system is
$$
\Pi(z,\e)=x(T,z,0)+\sum^{k}_{i=1}\e^{i}\dfrac{y_{i}(T,z)}{i!}+
\CO(\e^{k+1}).
$$
At $z=z(\e)$ this function has the Jacobian matrix
$$
\partial_z Pi(z(\e),\e)=\partial_z x(T,z,0)+\sum^{k}_{i=1}\dfrac{\e^{i}}{i!}
\partial_zy_{i}(T,z(\e))+\CO(\e^{k+1}).
$$
From Proposition \ref{p6}$(a)$ we have that
$z(\e)=\sum^{k-l}_{i=0}z_i+\CO(\e^{k-l})$. Thus, we calculate the
first $k-l$ coefficient of the Taylor expansion of the function
$$
J(\e)=\partial_z x(T,z(\e),0)+\sum^{k}_{i=1}\dfrac{\e^{i}}{i!}\partial_zy_{i}
(T,z(\e))+\CO(\e^{k+1}),
$$
where $J(0)=\partial_zy_{0}(T,z_0)$. By the Leibniz rule we have
\begin{align*}
\dfrac{d^j}{d\e^j}J(\e)&=\sum^{k}_{i=0}\dfrac{\e^{i}}{i!}
\dfrac{d^j}{d\e^j}\left(\partial_zy_{i}(T,z(\e))\right)+
\CO(\e^{k-j+1})\\
&=\sum^j_{i=0}\sum^j_{n=0}\binom{j}{n}\left(\e^i\right)^{n}
\dfrac{d^{j-n}}{d\e^{j-n}}\left(\partial_zy_{i}(T,z(\e))\right)
+\CO(\e)\Big|_{\e=0}.
\end{align*}
When $\e=0$ the only non-vanishing  terms in the above equation will
be those satisfying $i=n$, then we have
\begin{align}\label{est2}
\dfrac{d^j}{d\e^j}\partial_zd(z(\e),\e)=\sum_{n=0}^j\dfrac{j!}{n!(j-n)!}
\dfrac{d^{j-n}}{d\e^{j-n}}\left(\partial_zy_{n}(T,z(\e))
\right).
\end{align}
Using the  Fa\'{a} di Bruno's Formula again we have that
$$
\dfrac{d^i}{d\e^i}\left(\partial_zy_{n}(T,z(\e))
\right)\Big|_{\e=0}=\sum_{S_i}\dfrac{i!}{b_1!\cdots
b_i!(i-1)!^{b_i}}\partial^{I+1}_zy_{n}(T,z(\e))\bigodot^i_{u=1}
\left(\p^{u}z(\e)\right)^{b_u}\Big|_{\e=0}.
$$
Substituting the above equation in \eqref{est2} and taking $i=j-n$ we obtain
\begin{align*}
\ov{A}_j&=\dfrac{d^j}{d\e^j}\partial_zd(z(\e),\e)\Big|_{\e=0}\\
&=\sum^j_{i=0}\dfrac{j!}{i!(j-i)!}\sum_{S_i}\dfrac{i!}{b_1!\cdots
b_i!(i-1)!^{b_i}}\partial^{I+1}_zy_{j-i}(T,z_0)\bigodot^i_{u=1}
\left(u!z_u\right)^{b_u}.
\end{align*}
Finally we take $A_j=j!\ov{A}_j$. This completes the proof.
\end{proof}

\section{Application: Hopf bifurcation in $4$-dimensional polynomial system}\label{sec:app}
We use Theorem \ref{teoa} for studying the stability of a periodic solution resulting from a Hopf-Hopf bifurcation near resonance.
\begin{theorem}\label{teo:expA}
Consider the $4-$dimensional polynomial system 
\begin{equation}\label{expA}
\begin{array}{ccccc} 
\dot{x}&=& -\omega\, y&+&\e\, P_1(x,y,z)+\e^2 Q_1(x,y,z,b),\\ 
\dot{y}&=&\omega\, x&+&\e \,P_2(x,y,z)+\e^2 Q_2(x,y,z,c),\\ 
\dot{z}&=&-\omega\, w&+&\e \,P_3(x,y,z)+\e^2 Q_3(x,y,z,d),\\
\dot{w}&=&\omega\, z&+&\e\, P_4(x,y,z)+\e^2 Q_4(x,y,z,e).
\end{array} 
\end{equation}
Where $\omega=\frac{54 \pi }{7}$ and
\begin{align*}
&P_1(x,y,z)=\frac{1}{7} \left(2 w^2 (x+y+z)+w \left(2 x^2+2 y^2+2 z^2+1\right)+2 z \left(x^2+y^2\right)\right.\\
&\left.+2 z^2 (x+y)+2 x y (x+y)+x-28 y+z\right),\\ 
& \\
&Q_1(x,y,z,b)=b x \left(w^2+z^2\right)+w^2 (y+z)+2 w \left(x^2+y^2+z^2\right)+y \left(x^2+y z\right),\\
& \\
&P_2(x,y,z)=\frac{1}{7} \left(w^2 (2 x-5 y+2 z)+w \left(2 x^2+2 y^2+2 z^2+1\right)+2 z \left(x^2+y^2\right)\right.\\
&\left.+z^2 (2 x-5 y)+2 x y (x+y)+x+y+z\right),\\
& \\
&Q_2(x,y,z,c)=x^2 (c z+y+z)+y z (c y+2 z)+w^2 (x+y+z)+x \left(2 y^2+z^2\right),\\
& \\
&P_3(x,y,z)=\frac{1}{3052}\Big(w^2 (872 (y+z)-11719 x)+436 w \left(2 x^2+2 y^2+2 z^2+1\right)\\
&+872 x^2 (y+z)+x \left(436-24310 y^2-11719 z^2\right)+ (2 y (z (y+z)-14)+z)\Big),\\
& \\
&Q_3(x,y,z,d)=w^2 (d z+x+y)+d y^2 z+2 w \left(x^2+y^2+z^2\right)+x^2 y+x z^2,\\
& \\
&P_4(x,y,z)=w^2 \left(\frac{2 x}{7}+\frac{2021 y}{6104}+\frac{2 z}{7}\right)+\frac{1}{7} w \left(-34 x^2+2 y^2-34 z^2+1\right)\\
&+\frac{2}{7} x^2 (y+z)+\frac{2}{7} x \left(y^2+z^2\right)+\frac{1}{7} (x+y+z)+\frac{y z (1744 y+2021 z)}{6104},\\
& \\
&Q_4(x,y,z,e)=w^2 (e\, z+x+y)+x^2 (e\, z+y+z)+x \left(2 y^2+z^2\right)+y z (y+2 z).
\end{align*}
We also define
\begin{align*}
R(b,c,d,e)&=\frac{45308619078085195}{476090594062466256} b-\frac{72977802317402731}{476090594062466256} c\\
&+\frac{2947109871432167}{238045297031233128} d+\frac{3376284095107}{545975451906498}e\\
&+\frac{2185436011713349666848 \pi -3790812222100505261383}{9963623952539293805568 \pi }.
\end{align*} If $R(b,c,d,e)\neq0$ then, for $\e>0$ sufficiently small, this system has an isolated periodic solution bifurcating from the origin. Such a periodic solution is asymptotically stable provided that $R(b,c,d,e)<0$, and unstable provided that $R(b,c,d,e)>0$.
\end{theorem}

\begin{proof}
Using the cylindrical coordinates $x=\rho \cos(\theta)$, $y=\rho \sin(\theta)$, $z=r\cos(\theta+\al),$ and $w=r\sin(\theta+\al)$, we transform system \eqref{teo:expA} into a new differential system where  $\dot{\theta}=\omega+\CO(\e)$. Then,  taking $\theta$ as the new independent variable, we obtain a non-autonomous $2\pi$-periodic system in the standard form
\begin{equation}\label{eqA1}
\begin{split}
\dfrac{d\rho}{d \theta}=& \e F_{11}(\theta,\rho, r, \al)+\e^2 F_{21}(\theta,\rho, r, \al)+\CO(\e^3),\\
\dfrac{d r}{d \theta}=& \e F_{12}(\theta,\rho, r, \al)+\e^2 F_{22}(\theta,\rho, r, \al)+\CO(\e^3),\\
\dfrac{d \al}{d \theta}=&\e F_{13}(\theta,\rho, r, \al)+\e^2 F_{23}(\theta,\rho, r, \al)+\CO(\e^3),
\end{split}
\end{equation}
see the Appendix for the complete expressions. In order to apply Corollary \ref{cte}, we use the formulae \eqref{smoothyi} for computing the averaging function
\[
\bg_1(\rho,r,\al)=\left(g^1_1(\rho,r,\al),g^2_1(\rho,r,\al),g^3_1(\rho,r,\al)\right),
\]
where
\begin{align*}
g^1_1(\rho,r,\al)&=\frac{1}{54} \left(r \cos (\alpha ) \left(4 \rho ^2+r^2+2\right)+\rho  \left(\rho ^2-3 r^2+2\right)\right),\\
g^2_1(\rho,r,\al)&=\frac{1}{54} \left(29 \rho  \sin (\alpha )-8 r^3+9 \rho ^2 r \cos (2 \alpha )-14 \rho ^2 r+2 r\right)\\
&-\frac{\rho  \cos (\alpha )}{47088} \left(11719 \rho ^2+21417 r^2-1744\right),\\
g^3_1(\rho,r,\al)&=-\frac{29 \rho +r \left(r^2+2\right) \sin (\alpha )}{54 \rho }+\frac{\rho }{47088 r}\Big(\sin (\alpha ) \big(11719 \rho ^2+17929 r^2
\\
&-1744\big)-872 \cos (\alpha ) (18 \rho  r \sin (\alpha )-29)\Big).
\end{align*}
Thus, it is easy to see that $(\rho_0,r_0,\al_0)=\left(1,1,\dfrac{\pi}{2}\right)$ is a simple zero of the first averaging function, i.e.,
\[
\bg_1(\rho_0,r_0,\al_0)=0 \quad \mbox{and} \quad \det\left( D \bg_1(\rho_0, r_0, \al_0)\right)=-\dfrac{246677}{1271376}.
\]
Therefore, it follows from Corollary \ref{cte} that there exists a periodic solution $\f(t,\e)$ for system \eqref{eqA1}. In order to determine the stability of this solution we compute the second averaging function
\begin{equation*}\label{eqg2}
\bg_2(\rho,r,\al)=\left(g^1_2(\rho,r,\al),g^2_2(\rho,r,\al),g^3_2(\rho,r,\al)\right),
\end{equation*}
the complete expressions for the above functions can be found in the Appendix. Accordingly, we have the following expression for the initial condition of the periodic solution $\f(t,\e)$
\begin{align*}
\f(0,\e)&=(\rho_0, r_0,\al_0)+\e(\rho_1, r_1,\al_1)+\CO(\e^2) \quad \mbox{where}\\
 (\rho_1, r_1,\al_1)&=-\p\bg_1(\rho_0, r_0,\al_0)^{-1}\bg_2(\rho_0, r_0,\al_0).
\end{align*}
Which can be used for computing $\A(\e)=\CG_k(\f(0,\e),\e)$ obtaining
\[
\A(\e)=\e A_1+\e^2 A_2+\CO(\e^3),
\]
with
\begin{align*}
A_1&=\partial \bg_1(\rho_0, r_0,\al_0),\\
A_2&=\partial \bg_2(\rho_0, r_0,\al_0)+\partial ^2\bg_1(\rho_0, r_0,\al_0).(\rho_1, r_1,\al_1).
\end{align*}
Computing the eigenvalues of the jet $J_1(\e^{-1}\A(\e))=A_1+\e A_2$ we obtain
\begin{align*}
J_1(\lambda_1(\e))&= R(b,c,d,e)+\CO(\e)+i\left(\frac{1}{108}\sqrt{\frac{246677}{109}}+\CO(\e) \right),\\
J_1(\lambda_2(\e))&= R(b,c,d,e)+\CO(\e)-i\left(\frac{1}{108}\sqrt{\frac{246677}{109}}+\CO(\e) \right),\\
J_1(\lambda_3(\e))&=-1+\CO(\e).
\end{align*}
Consequently, it follows from Theorem \ref{teoa} that the periodic solution $\f(t,\e)$ will be asymptotically  stable if $R(b,c,d,e)<0$ and unstable if $R(b,c,d,e)>0$. The result then follows by going back through the change of variables.
\end{proof}

\begin{example}Consider system \eqref{expA} with $b=\dfrac{1}{250}$, $c=150$, $d=-1$ and $e=-1$. Then,  we have $R\left(\dfrac{1}{250},150,-1,-1\right)=-22.913$, consequently system \eqref{expA} has a stable periodic solution as shown in the Figure \eqref{exA}.
\begin{figure}[H]
	\begin{overpic}[width=12cm]{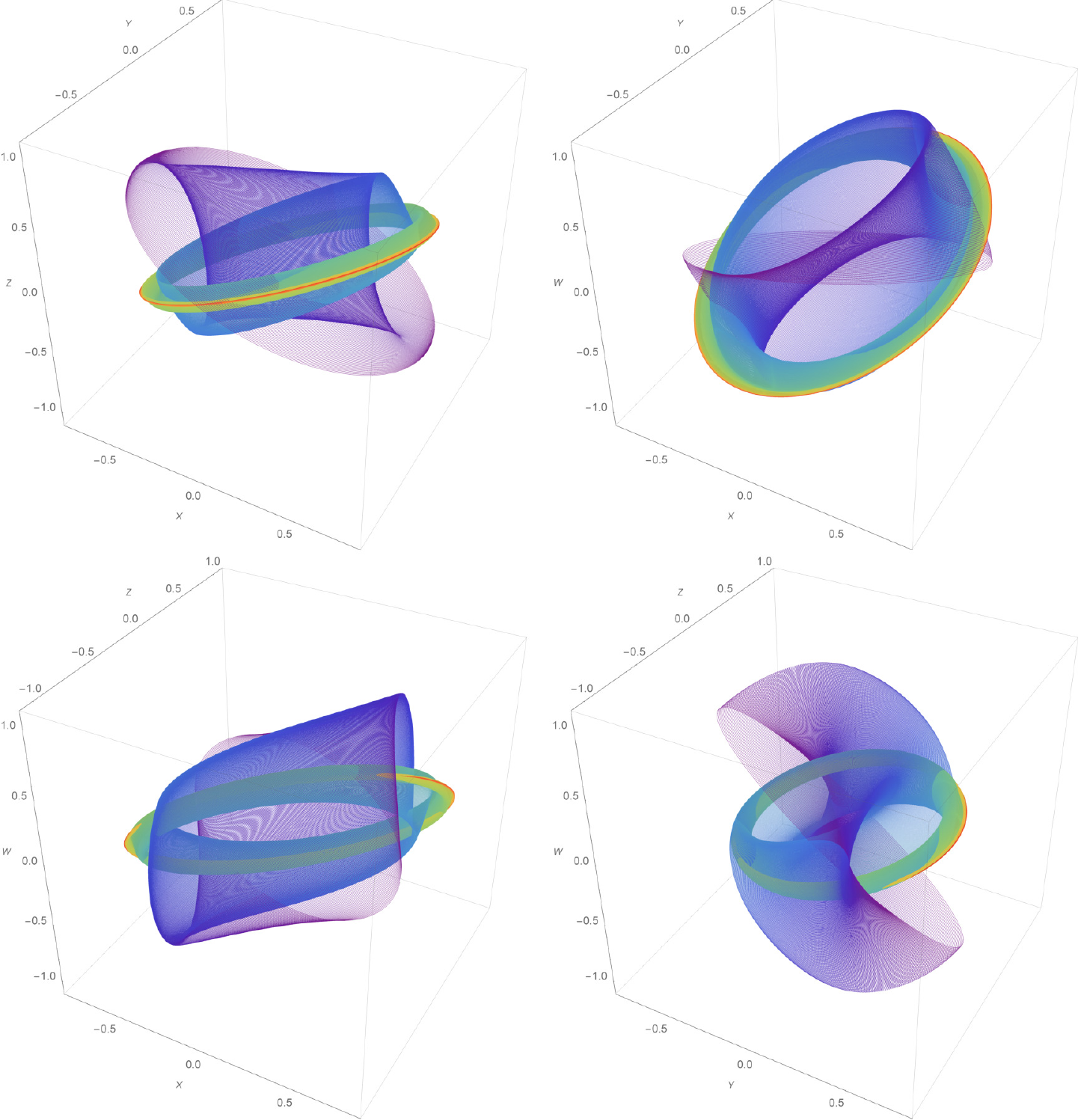}
	\end{overpic}
	\caption{Solution starting at $\left(0.7638,0,-0.4716,0.2242 \right)$ being atracted by the stable periodic solution detected by Theorem \ref{teo:expA}. Here, $\e=\dfrac{1}{45}$.}
		\label{exA}
\end{figure}
\end{example}

The next result is an application of Theorem \ref{teob} to a family of 4-dimensional differential systems. 
\begin{theorem}\label{texB} Consider the $4-$dimensional polynomial system
\begin{equation}\label{expb}
\begin{aligned} 
\dot{x}=& \, y+\e\left(\dfrac{2 x^3}{7}-x^2 z-x y^2+x z^2-\dfrac{91 x}{10}+\dfrac{2 z^3}{7}\right),\\ 
\dot{y}=&-\, x+\e\left(\dfrac{2 x^3}{7}-x^2 z-x y^2+x z^2+\dfrac{607 y}{70}+\dfrac{2 z^3}{7}\right),\\ 
\dot{z}=&\e\left(b \,y z+\dfrac{2 x^3}{7}-x^2 z-x y^2+x z^2+\dfrac{2 z^3}{7}+\dfrac{54 z}{7}\right),\\
\dot{w}=&- w+\e\left(\dfrac{2 x^3}{7}-x^2 z-x y^2+x z^2+\dfrac{2 z^3}{7}\right),
\end{aligned} 
\end{equation}
 where $\e>0$ and $b$ are real parameters. For $\e>0$ sufficiently small, this differential system has two isolated periodic solutions $\varphi_{\pm} (t,\e)$ bifurcating from $(4,0,\pm 1)$ respectively. Moreover, consider 
 $$
 R(b)=\dfrac{\pi  (16485\,b-122880 \pi -157337) }{3920}\neq 0.
 $$
If $R(b)>0,$ then $\varphi_+(t,\e)$ will be stable and $\varphi_-(t,\e)$ unstable. Otherwise, $\varphi_+(t,\e)$ will be unstable and $\varphi_-(t,\e)$ stable.
\end{theorem}

\begin{proof}
 Using the cylindrical coordinates $x=\rho \sin(\theta)$, $y=\rho \cos(\theta)$, we can transform system \eqref{expb} into a new system with $\dot{\theta}=1+\CO(\e)$. Then,  we can take $\theta$ as the new independent variable obtaining
\begin{equation}\label{Sint1}
\begin{split}
\dfrac{d\rho}{d \theta}=& \e F_{11}(\theta,\rho, z, w)+\e^2 F_{21}(\theta,\rho,z, w)+\CO(\e^3),\\
\dfrac{d z}{d \theta}=& \e F_{12}(\theta,\rho, z, w)+\e^2 F_{22}(\theta,\rho, z, w)+\CO(\e^3),\\
\dfrac{d w}{d \theta}=&- W+\e F_{13}(\theta,\rho, z, w)+\e^2 F_{23}(\theta,\rho, z, w)+\CO(\e^3),
\end{split}
\end{equation}
where  the functions 
$$
F_i(\theta,\rho,z, w)=\left(F_{i1}(\theta,\rho, z, w),F_{i2}(\theta,\rho,z, w),F_{i3}(\theta,\rho, z, w)\right)
$$ are smooth  and $2\pi$ periodic in $\theta$  for $i=1,2$.  Precisely, we have 
\begin{align*}
F_{11}(\theta,\rho, z, w)=&\left(-10 \sin (\theta ) \left(9 \rho ^3 \cos (3 \theta )+8 z^3\right)-4 \rho  \cos (2 \theta ) \left(10 \rho ^2+35 z^2-622\right)\right.\\
+&5 \rho  \left(\rho ^2-28 z^2+12\right)+10 \cos (\theta ) \left(\rho  \sin (\theta ) \left(\rho ^2-28 z^2+28 \rho  z \cos (\theta )\right)-8 z^3\right. \\
+&\left.\left.21 \rho ^2 z\right)+5 \rho ^2 (14 z \cos (3 \theta )-9 \rho  \cos (4 \theta ))\right)\dfrac{1}{280},\\
F_{12}(\theta,\rho,z, w)=&\dfrac{1}{28}\left(\rho  \left(\cos (\theta ) \left(\rho ^2-28 z^2\right)-28 b\, z \sin (\theta )-9 \rho ^2 \cos (3 \theta )+14 \rho  z \cos (2 \theta )\right)\right.\\
-&\left.8 z \left(z^2+27\right)+14 \rho ^2 z\right),\\
F_{13}(\theta,\rho, z, w)=&\frac{1}{280\rho}\left(10 \cos (\theta ) \left(\rho ^4-7 \rho ^2 (z (3 w+4 z))+8 w z^3\right)-45 \rho ^3 w \sin (4 \theta )\right.\\
+&45 \rho ^3 w \cos (4 \theta )+2 \rho  w \sin (2 \theta ) \left(25 \rho ^2-70 z^2+1244\right)+10 w z \sin (\theta ) \\
\times&\left(7 \rho ^2-8 z^2\right)+20 \rho  \cos (2 \theta ) \left(7 w z^2+\rho ^2 (2 w+7 z)\right)+20 \rho  z^2 (7 w-4 z)\\
+&\left.70 \rho ^2 w z \sin (3 \theta )-10 \rho ^2 \cos (3 \theta ) \left(9 \rho ^2+7 w z\right)-5 \rho ^3 (w-28 z)\right),
\end{align*}
The higher-order functions $F_{21}$,$F_{21}$, and $F_{21}$ have big expressions that will be omitted. The unperturbed part of system \eqref{expb} is
\[
\dfrac{d\rho}{d \theta}=0,\quad
\dfrac{d z}{d \theta}=0,\quad
\dfrac{d w}{d \theta}=- w,
\]
and clearly, its solution is given by
$$
x\left(\theta,(\rho,z,w),0\right)=\left( \rho, z,w\,e^{-\theta}\right).
$$
The unperturbed system has a 2-dimensional manifold of periodic solutions. Therefore,  the hypothesis {\bf H1} holds by taking $m=2$ and
$$
\CZ=\{z_\al=(\al_1,\al_2,0)\,: \, \al_1 \, \in\, (\al_1,\al_2)\in \ov V\},\quad V=  (\rho_0,\rho_1)\times(-z_0,z_0), \quad  \beta(\al_1,\al_2)=0.
$$
It will be sufficient to take $\rho_0=3,$ $\rho_1=5$, and $z_0=2.$
Moreover, it is easy to verify that {\bf H2} also holds. Indeed,  we have
$$
\bg_0(\rho,z,w)=\left( 0,\,0,\,\left(e^{2 \pi }-1\right) w\right),
$$
and the fundamental matrix becomes
\begin{align*}
\dfrac{\partial \bg_0}{\partial(\rho,z,w)}(2 \pi)=&Y_0(2 \pi,(\rho,z,w))-I_3=\begin{pmatrix}
 0 & 0 & 0 \\
 0 & 0& 0 \\
 0 & 0 & e^{-2 \pi }-1
\end{pmatrix}.
\end{align*}
We have directly that, $\det\Delta(\al)=e^{-2 \pi }-1\neq 0$.

In order to verify the hypothesis {\bf H3}, we apply Corollary \ref{1popct1a}.The formulae \eqref{eqgi} and \eqref{fi} as provide us the expressions
\begin{align*}
 \bg_1(\rho,z,w)=&Y_0(t,(\rho,z,w))\int_{0}^tY_0(s,(\rho,z,w))^{-1}F_1(s,x(s,(\rho,z,w),0))ds\Big|_0^{2\pi},\\
 f_1(\al)=&-\G(\al) \De(\al)^{-1}\pi^\perp \bg_1(z_\al)+\pi \bg_1(z_\al),
\end{align*}
where
$$
\G(\al)=\begin{pmatrix}0 \\ 0 \end{pmatrix} \quad \mbox{and} \quad \Delta(\al)=e^{2 \pi }-1.
$$
Consequently, the first-order averaging function is given by
$$
\bg_1(\rho,z,w)=\left(g_{11}(\rho,z,w),g_{12}(\rho,z,w),g_{13}(\rho,z,w) \right),
$$
with
\begin{align*}
g_{11}(\rho,z,w)=&\frac{1}{28} \pi  \rho  \left(\rho ^2-28 z^2+12\right),\\
g_{12}(\rho,z,w)=&-\frac{1}{7} \pi  z \left(-7 \rho ^2+4 z^2+108\right),\\
g_{13}(\rho,z,w)=&\frac{1}{140}\left(e^{2 \pi } \left(5 \pi  w \left(28 z^2-\rho ^2\right)-2 \left(\rho ^3+20 z^3+35 \rho  z^2-42 \rho ^2 z\right)\right)\right.\\
+&\left.2 \left(\rho ^3+20 z^3+35 \rho  z^2-42 \rho ^2 z\right)\right).
\end{align*}
On the other hand, the bifurcation function becomes
$$
f_1(\al_1,\al_2)=\left( \frac{1}{28} \pi \al_1  \left(\al_1^2-28 {\al_2}^2+12\right), \pi \al_1 ^2 {\al_2}-\frac{4}{7} \pi  {\al_2} \left({\al_2}^2+27\right)\right).
$$
Since $f_1$ has two simple  roots in $V$, 
$$
\al_{\pm}^*= (4,\pm 1) ,
$$ 
it follows from Corollary \ref{1popct1a} (taking $r=1$), that there exist two periodic solutions $\varphi_\pm(t,\e)$ for the  differential system \eqref{Sint1} such that $\varphi_\pm(0,\e)\rightarrow (4,\pm 1,0)$ as $\e \rightarrow 0$.

In what follows, we are going to provide the stability conditions of the periodic solution $\varphi_+(t,\e)$ by using Theorem \ref{teob}. An analogous reasoning ca be used to study $\varphi_-(t,\e)$. 

First, we observe that hypothesis {\bf H4} is satisfied by computing
\begin{equation}\label{IDG}
I_1+\Delta(\al_+^*)-\partial\beta(\al_+^*)\Gamma(\al_+^*)= e^{-2 \pi },
\end{equation}
which has no multiple eigenvalues on the unitary circle and $1$ is not one of its eigenvalues.
Moreover, we use the formulae provided by Proposition \ref{p6}, to obtain the expression of the initial condition
$$
\varphi_+(0,\e)=(4,1,0)+\e\left(\frac{575 \,b}{96}+\frac{50671}{3360},\frac{617 \, b}{96}+\frac{94457}{3360},\frac{4  }{5}\right)+\CO(\e^2). 
$$
Therefore, from \eqref{B}, \eqref{LYIL}, and \eqref{mp1} we have that $L=I_3$, $Y_{\e}(2\pi)=Y_{0}(2\pi)$, and
\begin{align*}
A(\e)&=\e\left(
\begin{array}{cc}
\dfrac{8 \pi }{7}&8 \pi\\
-8 \pi&-\dfrac{8 \pi }{7} 
\end{array}
\right)+\CO(\e^2),\\
B(\e)&=\e\left(\begin{array}{c}
0\\
0 
\end{array}
\right)+\CO(\e^2),\\
C(\e)&=\e\left(\begin{array}{cc}
\frac{25}{7} \left(1-e^{-2 \pi }\right) &  \frac{334}{35} \left(e^{-2 \pi }-1\right)
\end{array}
\right)+\CO(\e^2),\\
D(\e)&=-\e \dfrac{3 e^{-2 \pi }}{7}  +\CO(\e^2).
\end{align*}

From here, an taking \eqref{mp1} into account we get 
\begin{align*}
N(\e)&=
e^{-2 \pi }+\CO(\e),
\\
M(\e)&=\e\begin{pmatrix}
\dfrac{8 \pi }{7}&8 \pi\\
-8 \pi&-\dfrac{8 \pi }{7} 
\end{pmatrix}+\CO(\e^2).
\end{align*}
Notice that, from expression \eqref{Mep}, $\ell=1.$ Accordingly, taking $\mu_1=0$ and $\mu_2=1$ we have from \eqref{cor} and \eqref{est} that
\begin{align*}
J_0^\e\CP(\omega,\e)&=e^{-2 \pi }-\omega,\\
J_3^\e\CQ(\lambda,\e)&=\left(\lambda ^2+\frac{3072 \pi ^2}{49}\right) \e ^2+ \left(\frac{\pi  (336665\,b+184320 \pi +1509827) \lambda }{2940}\right.\\
+&\left.\frac{5}{343} \pi ^2 (2877 \,b-37297)\right)\e ^3.
\end{align*}

Notice that, $J_0^\e\CP(\omega,\e)$ does not have branches of zeros starting at the unitary circle and the eigenvalue of \eqref{IDG} belongs to the interior of the unitary disk. Then condition $(a1)$ of Theorem \ref{teob} is verified.

Now, computing the roots $\lambda_1(\e)$ and $\lambda_2(\e)$ of $J_3^\e\CQ(\lambda,\e)$, we obtain
\begin{equation*}
\begin{split}
J_1\lambda_1(\e)=&  \frac{32\pi\sqrt{3}}{7} i+\e\left(\frac{\pi  (16485\,b-122880 \pi -157337) }{3920}-\frac{ \pi  (47075 \,b+474583)  }{980 \sqrt{3}}i\right),\\
J_1\lambda_2(\e)=&  -\frac{32\pi\sqrt{3}}{7} i+\e\left(\frac{\pi  (16485\,b-122880 \pi -157337) }{3920}+\frac{ \pi  (47075 \,b+474583)  }{980 \sqrt{3}}i\right).
\end{split}
\end{equation*}
Since
$$
\Re(J_1\lambda_i(\e))=\e\dfrac{\pi  (16485\,b-122880 \pi -157337) }{3920}=\e R(b),
$$
the condition $(a2)$ of Theorem \ref{teob} is verified provided that $R(b)<0$, which implies the stability of the periodic solution $\varphi_+(t,\e)$. On the other hand, the condition $(b2)$ of Theorem \ref{teob} is verified provided that $R(b)>0$, in this case the periodic solution is unstable.

The same argument can be used for proving the statement about $\varphi_-(t,\e)$.
\end{proof}

\begin{example}Consider system \eqref{expb} with the coefficients $b=-1.1267$,  $\e=1/50$ and consequently $R(b)=-913.4$. Figure \eqref{figB} shows the three-dimensional projections  solutions starting at $\left(4.166, 0,1.417,-0.016\right)$ and $\left(4.021, 0,-1.501,0.228\right)$. These solutions are attracted by the stable periodic solution $\varphi_+(t,\e)$ and leaves the neighborhood of the unstable periodic solution $\varphi_-(t,\e)$ respectively.

\begin{figure}[H]
	\begin{overpic}[width=10cm]{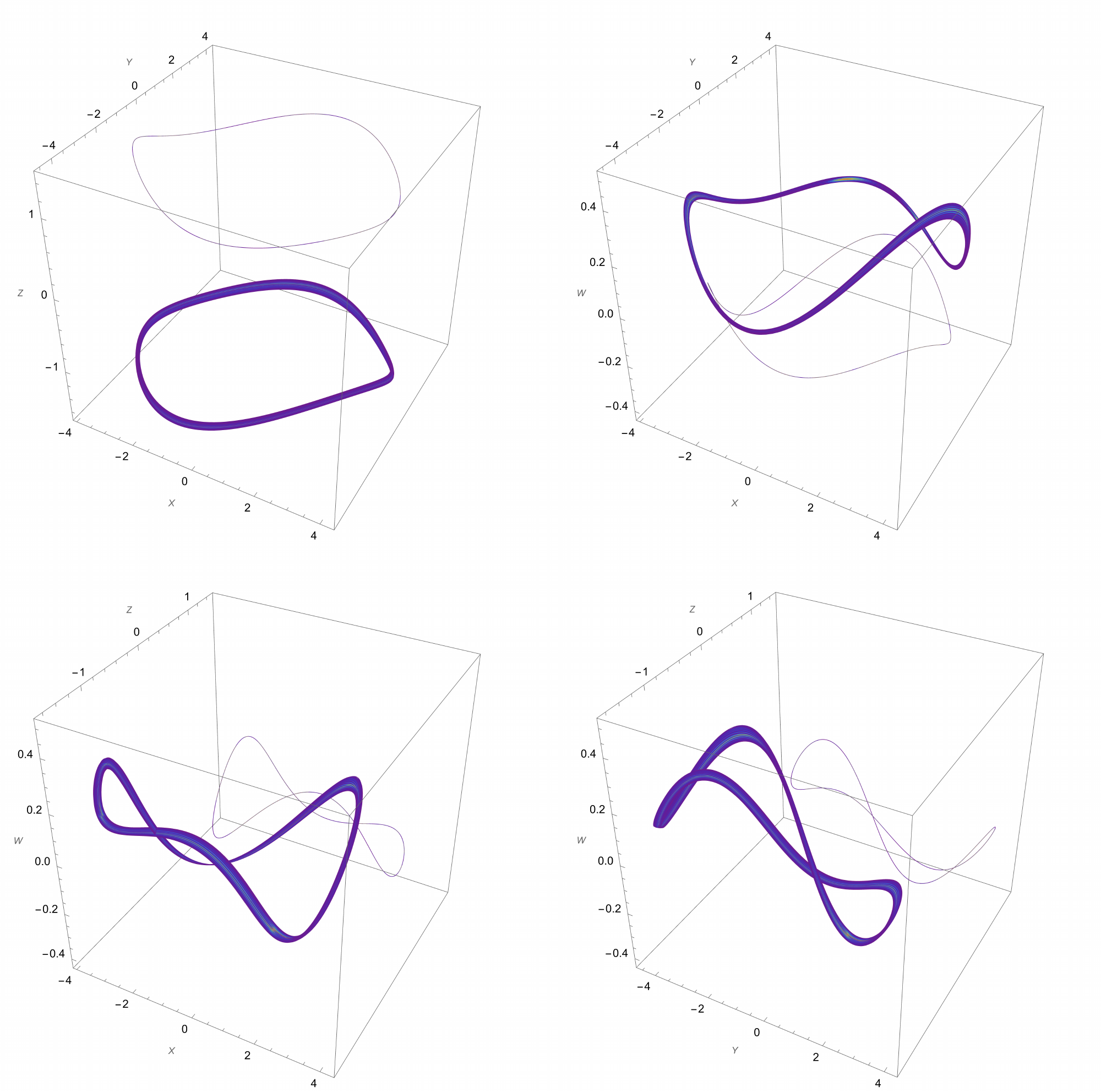}
	\end{overpic}
	\caption{Stable and Unstable periodic solutions in system \eqref{expb}.}
	\label{figB}
\end{figure}
\end{example}

\section*{Acknowledgements}

DDN is partially supported by S\~{a}o Paulo Research Foundation (FAPESP) grants 2022/09633-5, 2019/10269-3, and 2018/13481-0, and by Conselho Nacional de Desenvolvimento Cient\'{i}fico e Tecnol\'{o}gico (CNPq) grants 438975/2018-9 and 309110/2021-1. 

\bibliographystyle{abbrv}
\bibliography{references}

\end{document}